\documentclass[a4paper, 11pt, reqno]{amsart}

\usepackage{amsfonts, amsthm, amssymb, amsmath, stmaryrd}
\usepackage{mathrsfs,array}
\usepackage{eucal,fullpage,times,color,enumerate,accents}
\usepackage{bbm}
\usepackage{url}
\usepackage{cancel}
\usepackage{tikz}
\usetikzlibrary{fadings}
\tikzfading[name=fade out,
  inner color=transparent!0, outer color=transparent!100]
\tikzfading[name=fade right,
  left color=transparent!0, right color=transparent!100]
\tikzfading[name=fade left,
  right color=transparent!0, left color=transparent!100]
\tikzfading[name=fade mid,
  left color = transparent!100, right color = transparent!100, middle color=transparent!0]
\usepackage{ifthen}
\tikzset{
  laser beam action/.style={
    line width=\pgflinewidth+1.4pt,draw opacity=.095,draw=#1,
  },
  laser beam recurs/.code 2 args={%
    \pgfmathtruncatemacro{\level}{#1-1}%
    \ifthenelse{\equal{\level}{0}}%
    {\tikzset{preaction={laser beam action=#2}}}%
    {\tikzset{preaction={laser beam action=#2,laser beam recurs={\level}{#2}}}}
  },
  laser beam/.style={preaction={laser beam recurs={20}{#1}},draw opacity=1,draw=#1},
}
\usepackage{graphics}
\usepackage{graphicx}
\usepackage{xypic}
\usepackage{bm}
\usetikzlibrary{calc}
\usepackage[font=small,labelfont=bf]{caption}

\definecolor{navy}{rgb}{0,0,.65}

\usepackage[colorlinks,backref]{hyperref}
\hypersetup{colorlinks=true,urlcolor=blue,linkcolor=navy,citecolor=navy}

\makeatletter
\def\@tocline#1#2#3#4#5#6#7{\relax
  \ifnum #1>\c@tocdepth % then omit
  \else
    \par \addpenalty\@secpenalty\addvspace{#2}%
    \begingroup \hyphenpenalty\@M
    \@ifempty{#4}{%
      \@tempdima\csname r@tocindent\number#1\endcsname\relax
    }{%
      \@tempdima#4\relax
    }%
    \parindent\z@ \leftskip#3\relax \advance\leftskip\@tempdima\relax
    \rightskip\@pnumwidth plus4em \parfillskip-\@pnumwidth
    #5\leavevmode\hskip-\@tempdima
      \ifcase #1
       \or\or \hskip 1em \or \hskip 2em \else \hskip 3em \fi%
      #6\nobreak\relax
    \dotfill\hbox to\@pnumwidth{\@tocpagenum{#7}}\par
    \nobreak
    \endgroup
  \fi}
\makeatother

\newtheorem{theorem}{Theorem}[subsection]
\newtheorem{corollary}[theorem]{Corollary}
\newtheorem{lemma}[theorem]{Lemma}
\newtheorem{proposition}[theorem]{Proposition}
\newtheorem{definition}[theorem]{Definition}

\newtheorem{conjecture}[theorem]{Conjecture}

\newtheorem{quasi-theorem}[theorem]{Quasi-Theorem}
\newtheorem{blank remark}[theorem]{}

\newtheorem{Th}{Theorem}[]

\newtheorem{rem1}[theorem]{Remark}
\newenvironment{remark}{\begin{rem1}\em}{\end{rem1}}

\newtheorem{not1}[theorem]{Notation}

\newcommand{\A}{{\mathbb{A}}}

\newcommand{\PP}{\mathbb{P}}         
\newcommand{\QQ} {{\mathbb Q}}		
\newcommand{\RR} {{\mathbb R}}		
\newcommand{\ZZ} {{\mathbb Z}}		
	
\newcommand{\FF}{{\mathbb F}}

\DeclareMathOperator{\Gal}{Gal}
\DeclareMathOperator{\sSpec}{Spec_{\ \!}}

\DeclareMathOperator{\Div}{div}

\newcommand{\lra}{{\longrightarrow}}
\DeclareMathOperator{\Def}{\overset{{}_{\text{def}}}{=}}

\newsavebox{\foobox}

\newcommand{\mono}{\!\xymatrix{{}\ar@{^{(}->}[r]&{}}\!}
\newcommand{\epi}{\!\xymatrix{{}\ar@{->>}[r]&{}}\!}

\newcommand{\young}{\scalebox{.7}{$\pmb{\square\!\square}${\larger\larger $\pmb{\cdot\!\cdot\!\cdot}$}$\pmb{\square}$}}

\newcommand{\lilF}{\mbox{{\smaller\smaller\smaller\smaller\smaller $\overline{\FF_{\!q}\!}$}}}
  
\begin{document}

\title{Primes in short intervals on curves over finite fields}

\author{Efrat Bank}
\address{\newline {\bf Efrat Bank} \newline
University of Michigan Mathematics Department \newline
530 Church Street \newline
Ann Arbor, MI 48109-1043 \newline
United States}
\email{\href{mailto:ebank@umich.edu}{ebank@umich.edu}}

\author{Tyler Foster}
\address{\newline {\bf Tyler Foster} \newline
Max Planck Institute for Mathematics \newline
Vivatsgasse 7 \newline
53111 Bonn \newline
Germany}
\email{\href{mailto:foster@mpim-bonn.mpg.de}{foster@mpim-bonn.mpg.de}}

\date{\today}

\pagestyle{plain}

\maketitle
\maketitle

\begin{abstract}
	We prove an analogue of the Prime Number Theorem for short intervals on a smooth projective geometrically irreducible curve of arbitrary genus over a finite field. A short interval ``of size $E$" in this setting is any additive translate of the space of global sections of a sufficiently positive divisor $E$ by a suitable rational function $f$. Our main theorem gives an asymptotic count of irreducible elements in short intervals on a curve in the ``large $q$" limit, uniformly in $f$ and $E$. This result provides a function field analogue of an unresolved short interval conjecture over number fields, and extends a theorem of Bary-Soroker, Rosenzweig, and the first author, which can be understood as an instance of our result for the special case of a divisor $E$ supported at a single rational point on the projective line.
\end{abstract}

\setcounter{tocdepth}{1}

\tableofcontents

%%%%%%%%%%%%%%%%%%%%%%%%%%%%%%%%%%%%%%%
%%%%%%%%%%%%%%%%%%%%%%%%%%%%%%%%%%%%%%%

\begin{section}{Introduction}
	
	In this paper, we give an asymptotic count of irreducible elements inside short intervals on a smooth projective geometrically irreducible curve over a finite field. Our main result (\S\ref{subsec: number fields}, Theorem \ref{theorem: main theorem}) provides a function field analogue of an unresolved short interval conjecture for number fields (Conjecture \ref{conj: prime numbers in SI, number field 3}), and extends a short interval theorem of Bary-Soroker, Rosenzweig, and the first author \cite[Corollary 2.4]{BBR15} for polynomials over finite fields.
	
	The notion of short intervals on a curve which we use is a natural analogue of the familiar notion of short intervals over the integers. In this introduction, we review what is known about short intervals over the integers, over number fields, and over polynomials with coefficients in a finite field. The analogies that run between these different settings lead naturally to our definition of a short interval on a curve and to the statement of our main result.
	
\begin{subsection}{The Prime Number Theorem for short intervals}\label{subsec: Primes in SI}
	The Prime Number Theorem (PNT) states that the asymptotic density of prime integers in real intervals $(0,x]$ is $1/\log x$. In other words, if we let $\pi(x)$ denote the prime counting function
	$$
	\pi(x)
	\ =\ 
	\# \{ 0<p\leq x  :  p \mbox{ is a prime integer}\},
	$$
then 
	\begin{equation}\label{eq: PNT}
	\pi(x)\sim \frac{x}{\log x}
	\ \ \ \ \ \ \mbox{as}\ \ \ \ \ \ 
	x\to \infty.
	\end{equation}
We get more refined statements by considering the asymptotic density of primes in families of smaller intervals. Letting $\Phi(x)$ be a real valued function with $0<\Phi(x)<x$, we can ask for the density of primes in the intervals $I(x,\Phi)\Def\big[x-\Phi(x),x+\Phi(x)\big]$ as $x\to\infty$. Define
	$$
	\begin{array}{rcl}
	\pi\big(I(x,\Phi)\big)
	&
	\!\!\Def\!\!
	&
	\#\big\{p\in I(x,\Phi):p \mbox{ is a prime integer}\big\},
	\end{array}
	$$
Then the naive conjecture on the asymptotic density of primes in the intervals $I(x,\Phi)$ is
	\begin{equation}\label{eq:PNT interval}
	\pi\big(I(x,\Phi)\big)\ \sim\ \frac{\#\big(I(x,\Phi)\big)}{\log x}
	\ \ \ \ \ \ \mbox{as}\ \ \ \ \ \ 
	x\to \infty.
	\end{equation}
For fixed $0<c<1$ and $\Phi(x)\sim c\ \!x$, it is a straightforward consequence of the PNT that \eqref{eq:PNT interval} holds. Assuming the Riemann hypothesis, \eqref{eq:PNT interval} holds for $\Phi(x)=x^{\varepsilon+\mbox{{\smaller\smaller\smaller\smaller\smaller $\frac{1}{2}$}}}$ for small $0<\varepsilon<1/2$. On the other hand, Maier \cite{Maier} established what is now known as the ``Maier phenomenon": for $\Phi(x)=(\log\ x)^{A}$, with $A>1$, the asymptotic formula \eqref{eq:PNT interval} fails. A classical conjecture predicts the following ``short interval" prime number conjecture:
	
\begin{conjecture}\label{conj: prime numbers in SI}
\normalfont
		For $0<\varepsilon<1$ and $\Phi(x)=x^{\varepsilon}$, the asymptotic formula \eqref{eq:PNT interval} holds. 
\end{conjecture}

\noindent
	In its full generality, Conjecture \ref{conj: prime numbers in SI} is still open. Heath-Brown \cite{HeathBrown1988Crelle}, improving on Huxley \cite{Huxley1972inv}, proved Conjecture \ref{conj: prime numbers in SI} for $\frac{7}{12}\le\varepsilon<1$. We refer the reader to the surveys of Granville \cite{Granville1995, Granville2010} for additional background.
	
\end{subsection}

%%%%%%%%%%%%%%%%%%%%%%%%%%%%%%%%%%%%%%%
	
\begin{subsection}{The Prime Polynomial Theorem for short intervals over $\pmb{\FF_{\!q}[t]}$}\label{subsec: Poly SI}
   	For each finite field $\FF_{\!q}$, the analogy between number fields and function fields provides us with the following table of corresponding sets and quantities:
	\begin{equation}\label{eq: table of analogies}
	{\renewcommand{\arraystretch}{1.5}
	\begin{array}{c|c}
	\pmb{\ZZ}
	&
	\mbox{ {\bf ring of polynomials}  }\pmb{\FF_{\!q}[t]}
	\\\hline\hline
	|x|
	&
	|f|\ \Def\ q^{\text{deg}_{\ \!}f}
	\\\hline
	(0,x]
	&
	M(k,q)\overset{{}_{\text{def}}}{=}\{h\in \FF_{\!q}[t]:h \mbox{ is monic and } \deg{h}= k\}
	\\\hline
	x=\#(0,x]
	&
	q^k=\#M(k,q)
	\\ \hline
	\log x
	&
	k\ =\ \log_{q} q^{k}
	\end{array}
	}
	\end{equation}
If we let $\pi_{q}(k)$ denote the prime polynomial counting function
	$$
	\pi_q(k)\ =\ \# \{h\in M(k,q) : h \mbox{ is irreducible}\},
	$$
then, in accord with Table \eqref{eq: table of analogies}, the Prime Polynomial Theorem (PPT) 
asserts that
	\begin{equation}\label{eq: PPT}
	\pi_q(k)
	\ \sim\ 
	\frac{q^k}{k}
	\ \ \ \ \ \ \mbox{as}\ \ \ \ \ \ 
	q^{k}\to\infty.
	\end{equation}
Table \eqref{eq: table of analogies} also suggests a natural definition of short intervals in $\FF_{\!q}[t]$:
	
\begin{definition}\label{definition: polynomial SI}
\normalfont
	Given any monic non-constant polynomial $f\in\FF_q[t]$ and any positive real number $\varepsilon$, the corresponding \textit{interval} ({\em around $f$}) is the set
	$$
	I(f,\varepsilon)
	\ \Def\ 
	\{h\in\FF_q[t] : | h-f| \leq | f |^{\varepsilon}   \}.
	$$
If $m\Def\lfloor \varepsilon\deg f \rfloor$ and $\FF_{\!q}[t]^{\le m}$ denotes the space of polynomials of degree at most $m$, then $I(f,\varepsilon)=f+\FF_{\!q}[t]^{\le m}$. We say that $I(f,\varepsilon)$ is a \textit{short interval} if $\varepsilon<1$, i.e., if $m<\deg f$.
\end{definition}
\begin{remark}\label{monicToSI}
Note that in view of Definition \ref{definition: polynomial SI}, the set $M(k,q)$ of monic polynomials of degree $k$ is the short interval $I(t^k,k-1)$. 
\end{remark}
	
	Initial results on the density of prime polynomials in short intervals can be deduced from the work of Cohen \cite{Cohen1972} when $\text{char}_{\ \!}\FF_q>\text{deg}_{\ \!}f$, and from the work of Keating and Rudnick \cite{KeatingRudnick2012} in an almost everywhere sense. In \cite{BBR15}, the first author together with Bary-Soroker and Rosenzweig prove the following analogue of Conjecture~\ref{conj: prime numbers in SI} in the large $q$ limit:
	
\begin{theorem}\label{theorem: main from BBR}
\normalfont
	{\bf \cite[Corollary 2.4]{BBR15}.}
	For fixed $k>0$ and a monic polynomial $f\in\FF_{\!q}[t]$ satisfying $\text{deg}_{\ \!}f=k$ and $\varepsilon>0$, define
	$$
	\pi_{q}\big(I(f,\varepsilon)\big)
	\ \ \Def\ \ 
	\#\big\{h\in I(f,\varepsilon):h\mbox{ is a prime polynomial}\big\}.
	$$
Then the asymptotic formula
	\begin{equation}\label{equation: asymptotic formula from BBR}
	\pi_q\left(I(f,\varepsilon) \right)\sim \frac{\#I(f,\varepsilon)}{k}
	\ \ \ \ \ \ \mbox{as}\ \ \ \ \ \ 
	q\to\infty
	\end{equation}
holds uniformly for all monic $f\in\FF_q[t]$ of degree $k$ and all
	\begin{equation}\label{equation: main from BBR conditions}
	\varepsilon_0\ \le\ \varepsilon\ <\ 1,
	\ \ \ \ \ \ \mbox{where}\ \ \ \ \ \ 
	\varepsilon_0= \begin{cases}
	\ \frac{3}{k} & \mbox{if $\text{char}_{\ \!}\FF_{\!q}=2$ and $f'$ is constant;}\\[4pt] 
	\ \frac{2}{k} & \mbox{if $\text{char}_{\ \!}\FF_{\!q}\ne 2$ or $f'$ is non-constant;}\\[4pt]
	\ \frac{1}{k} & \mbox{if $\text{char}_{\ \!}\FF_{\!q}\nmid k(k-1)$.}
	\end{cases}
	\end{equation}
\end{theorem}
	
\end{subsection}

\begin{subsection}{Short interval conjectures over number fields}\label{subsec: number fields}
	One can extend Conjecture \ref{conj: prime numbers in SI} to arbitrary number fields. However, because the relevant notions in $\QQ$ have several competing generalizations to number fields larger than $\QQ$, there are several competing generalizations of Conjecture \ref{conj: prime numbers in SI}. If we let $K$ be an algebraic number field of degree $n$ over $\QQ$, with ring of integers $\mathcal{O}_{K}$, then each prime $\mathfrak{a}\subset\mathcal{O}_{K}$ comes with a norm $N_{K}(\mathfrak{a})\Def\#(\mathcal{O}_{K}/\mathfrak{a})$. The norm of an element $a\in\mathcal{O}_{K}$ is by definition the norm of the ideal that $a$ generates. We have both a prime ideal and a principal prime ideal counting function:
	$$
	\begin{array}{rcl}
	\pi_{K}(x)
	&
	\!\!=\!\!
	&
	\# \{ \mbox{prime\ ideals\ }\mathfrak{p}\subset\mathcal{O}_{K} : \ \ 2<N_{K}(\mathfrak{p})\leq x\};
	\\[6pt]
	\pi_{K,\text{prin}}(x)
	&
	\!\!=\!\!
	&
	\# \{ \mbox{principal prime\ ideals\ }(a)\subset\mathcal{O}_{K} :\ \   2<N_{K}(a)\leq x\}.
	\end{array}
	$$
Landau's Prime Ideal Theorem (PIT) \cite{LandauPIT} states that
	\begin{equation}
	\pi_{K}(x)\ \sim \frac{x}{\log x}
	\ \ \ \ \ \ \mbox{as}\ \ \ \ \ \ 
	x\to\infty.
	\end{equation}
Letting $h_{K}$ denote the class number of $K$, the Principal PIT \cite[\S 7.2, Corollary~4]{Nark} states that
	\begin{equation}
	\pi_{K,\text{prin}}(x)\ \sim\ \frac{1}{h_{K}}\cdot\frac{x}{\log x}
	\ \ \ \ \ \ \mbox{as}\ \ \ \ \ \ 
	x\to\infty.
	\end{equation}
As before, we can attempt to refine these density theorems by considering any real valued function $\Phi(x)$, with $0<\Phi(x)<x$, the corresponding (real) intervals $I(x,\Phi)=[x-\Phi(x),x+\Phi(x)]$ and the prime ideal counting function
	$$
	\pi_{K}\big(I(x,\Phi)\big)
	\ =\ 
	\# \big\{\mbox{primes }\mathfrak{p}\subset\mathcal{O}_{K}:x-\Phi(x)\le N_{K}(\mathfrak{p})\le x+\Phi(x)\big\}.
	$$
The naive guess about the asymptotic behavior of $\pi_{K}\big(I(x,\Phi)\big)$ is that
	\begin{equation}\label{eq:PIT interval}
	\pi_{K}\big(I(x,\Phi)\big)\ \sim\ \frac{\#I(x,\Phi) }{\log x}=\ \frac{\ 2\ \!\Phi(x)\ }{\log x}
	\ \ \ \ \ \ \mbox{as}\ \ \ \ \ \ 
	x\to \infty.
	\end{equation}
When $\Phi(x)\sim c x$ for fixed $0<c<1$, formula \eqref{eq:PIT interval} follows directly from the PIT. Balog and Ono \cite{BO}, using formulas for the prime ideal counting function due to Lagarias and Odlyzko \cite{LO} and zero density estimates for Dedekind zeta-functions due to Heath-Brown \cite{HB77} and Mitsui \cite{mitsui1968}, show that formula \eqref{eq:PIT interval} holds for $x^{1-\mbox{{\smaller\smaller\smaller\smaller $\frac{1}{c}$}}+\varepsilon}\leq \Phi(x) \leq x$. Here one may take $c=8/3$ if $[K:\QQ]=2$, and one can take $c=[K:\QQ]$ if the degree of the extension is at least $3$. Assuming the Riemann Hypothesis for the Dedekind zeta function $\zeta_{K}(s)$, Greni\'{e}, Molteni, and Perelli \cite{Perelli} show that \eqref{eq:PIT interval} holds for all $\Phi(x) = \big(n\log x + \log |\text{disc}(K)|\big)\sqrt{x}$. 

	In a general number field, the norm $N_{K}(a)$ of an element $a\in\mathcal{O}_{K}$ is not equal to the absolute value of $a$ at a single infinite place of $K$. Likewise, given an element $b\in\mathcal{O}_{K}$ with $x\Def N_{K}(b)$, and given $0<\varepsilon<1$, the set of all $a\in\mathcal{O}_{K}$ satisfying $\big|N_{K}(a)-x\big|\le x^{\varepsilon}$ (with $|\cdot |$ the absolute value in $\RR$) is not necessarily the same as the set of all $a\in\mathcal{O}_{K}$ satisfying $N_{K}(a-b)\le x^{\varepsilon}$. This ambiguity in generalizing the basic quantities in Conjecture \ref{conj: prime numbers in SI} gives us at least two distinct conjectures that can be seen as extensions Conjecture \ref{conj: prime numbers in SI} to an arbitrary number field $K$:
	
\begin{conjecture}\label{conj: prime numbers in SI, number field 3}
\normalfont
	Let $S=\{\mbox{infinite places of }K\}$. There exists some constant $c$ such that for each real vector $\varepsilon_{S}=(\varepsilon_{\mathfrak{p}})_{\mathfrak{p}\in S}$ in $(0,1)^{S}\subset\RR^{S}$, the count
	$$
	\pi_{K,\text{prin}}\big(I(b,\varepsilon_{S})\big)
	\ =\ 
	\#\big\{a\in\mathcal{O}_{K}:|a-b|_{\mathfrak{p}}\le |b|^{\varepsilon_{\mathfrak{p}}}_{\mathfrak{p}}\mbox{ for each }\mathfrak{p}\in S,\mbox{ and }(a)\subset\mathcal{O}_{K}\mbox{ is prime}\big\}
	$$
satisfies the asymptotic formula
	\begin{equation}
	\pi_{K,\text{prin}}\big(I(b,\varepsilon_{S})\big)
	\ \ \sim\ \ 
	c\cdot\frac{\#\{a\in\mathcal{O}_{K}:|a-b|_{\mathfrak{p}}\le |b|^{\varepsilon_{\mathfrak{p}}}_{\mathfrak{p}}\mbox{ for all }\mathfrak{p}\in S\}}{\log N_{\!K\!}(b)}
	\ \ \ \ \ \ \mbox{as}\ \ \ \ \ \ 
	N_{\!K\!}(b)\to\infty.
	\end{equation}
\end{conjecture}
\begin{conjecture}\label{conj: prime numbers in SI, number field 1}
	\normalfont
	There exists some constant $c$ such that for each $0<\varepsilon<1$, the count
	$$
	\pi_{K,\text{prin}}\big(I(x,\varepsilon)\big)
	\ =\ 
	\#\big\{\mbox{principal prime ideals }(a)\subset\mathcal{O}_{K}:\ \ x-x^{\varepsilon}<N_{K}(a)\le x+x^{\varepsilon}\big\}.
	$$
	satisfies the asymptotic formula
	\begin{equation}
	\pi_{K,\text{prin}}\big(I(x,\varepsilon)\big)
	\ \ \sim\ \ 
	c\cdot\frac{\#I(x,\varepsilon)}{\log x}
	\ \ =\ \ 
	c\cdot\frac{\ 2\ x^{\varepsilon}}{\log x}
	\ \ \ \ \ \ \mbox{as}\ \ \ \ \ \ 
	x\to\infty.
	\end{equation}
\end{conjecture}

\begin{remark}
	For $K=\QQ$, Conjectures \ref{conj: prime numbers in SI, number field 3} and \ref{conj: prime numbers in SI, number field 1} both recover Conjecture \ref{conj: prime numbers in SI} if $c=1$.
\end{remark}

\end{subsection}
%%%%%%%%%%%%%%%%%%%%%%%%%%%%%%%%%%%%%%

%%%%%%%%%%%%%%%%%%%%%%%%%%%%%%%%%%%%%%%
	
\begin{subsection}{Main result: short intervals on arbitrary curves over $\pmb{\FF_{\!q}}$}\label{subsection: short intervals on arbitrary curves}
	Let $C$ be a smooth projective geometrically irreducible curve over $\FF_q$. As shown in \cite[Theorem 5.12]{Rosen}, the natural analogue of the PNT holds on $C$, which is to say that the counting function
	$$
	\pi_{C}(k)\ \Def\ \# \{ P \mbox{ a prime divisor of } C : \deg(P)=k\}
	$$
satisfies the asymptotic formula
	$$
	\pi_{C}(k) \sim \frac{q^k}{k}
	\ \ \ \ \ \ \mbox{as}\ \ \ \ \ \ 
	q^k\to \infty.
	$$
One can formulate analogues of each of the Conjectures \ref{conj: prime numbers in SI, number field 3} and \ref{conj: prime numbers in SI, number field 1} on $C$. In the present paper, we focus our attention to the analogue of Conjecture \ref{conj: prime numbers in SI, number field 3} in the large $q$ limit. We intend to address analogues of Conjecture \ref{conj: prime numbers in SI, number field 1} in a future paper.
	
	On $C$, the natural analogue of the ``short interval" implicit in Conjecture \ref{conj: prime numbers in SI, number field 3} is the following set:
	
\begin{definition}\label{definition: general SI}\normalfont
		 Let $E=m_1\mathfrak{p}_{1}+\cdots +m_s\mathfrak{p}_{s}$ be an effective divisor on $C$, and let $f$ be a regular function on the complement of $E$. The {\em interval} ({\em of size $E$ around $f$}) is the set
			\begin{equation}\label{equation: polynomial-like case}
			\ \ \ \ \ \ \ \ \ 
			\begin{array}{rcl}
			I(f,E)
			&
			\!\!\Def\!\!
			&
			\left\{\!\!
			\begin{array}{c}
			\mbox{regular functions $h$ on }C\!\smallsetminus\! \text{supp}(E) \ \text{such}
			\\
			\mbox{that }\nu_{\mathfrak{p}_{i}}(h-f)\ge -m_{i}\ \mbox{ for all }\ 1\le i\le s
			\end{array}
			\!\!\right\}
			\\[15pt]
			&
			\!\!=\!\!
			&f+H^{0}\big(C,\mathscr{O}(E)\big),
			\end{array}
			\end{equation}
where $H^{0}\big(C,\mathscr{O}(E)\big)$ is the space of regular functions on $C\!\smallsetminus\!\{\mathfrak{p}_{1},\dots,\mathfrak{p}_{s}\}$ with a pole of order at most $m_{i}$ at each point $\mathfrak{p}_{i}$, for $1\le i\le s$.
	
	The interval $I(f,E)$ is a \textit{short interval} if the order of the pole of $f$ at each $\mathfrak{p}_{i}$ is strictly greater than $m_i$.
\end{definition} 
	
\begin{remark}
	When $C=\PP^{1}$ and $E=m\ \infty$, for $m> 0$, Definitions \ref{definition: general SI} and \ref{definition: polynomial SI} coincide.
	
	The value that serves as our prime count in any short interval $I(f,E)$ is
	\begin{equation}
	\pi_{C}(I(f,E))
	\ \Def\ 
	\#\left\{\!\!
	\begin{array}{c}
	\text{$h\in I(f,E)$ such that $h$ generates a}
	\\
	\text{prime ideal in the ring of regular}
	\\ 
	\text{functions on }  C\!\smallsetminus\! \text{supp}(E)
	\end{array}
	\!\!\right\}.
	\end{equation}
The central result of the present paper is the following theorem, which establishes a function field analogue of Conjecture \ref{conj: prime numbers in SI} and its generalization \ref{conj: prime numbers in SI, number field 3}.  In addition, this result extends Theorem \ref{theorem: main from BBR} to curves of arbitrary genus over $\FF_{\!q}$:
\end{remark}

\begin{Th}\label{theorem: main theorem}
		\normalfont
		Let $C$ be a smooth projective geometrically irreducible curve of genus $g$ over $\FF_{\!q}$. Fix a positive integer $k>0$. Let $E$ be an effective divisor on $C$, and let $f$ be a regular function on $C\!\smallsetminus\! \text{supp}(E)$ such that the sum of the orders of all poles of $f$ is equal to $k$, and such that $I(f,E)$ is a short interval. Assume that either
		\begin{itemize}
			\item[{\bf (i)}] \vskip .1cm
			$E\geq 3E_0$ for some effective divisor $E_0$ with $\text{deg}_{\ \!}E_{0}\geq 2g+1$, or
			\item[{\bf (ii)}] \vskip .2cm
			$\text{char}_{\ \!}\FF_{\!q}= 2$, $E\geq 2E_0$  for some effective divisor $E_0$ with $\text{deg}_{\ \!}E_{0}\geq 2g+1$, such that the differential $df$ vanishes on a nonempty finite subset of $C\!\smallsetminus\! \text{supp}(E)$.\vskip .2cm
		\end{itemize}
	Then
		\begin{equation}\label{equation: asymptotic formula in main theorem}
		\pi_{C}\big(I(f,E)\big)
		\ =\ 
		\frac{\#I(f,E)}{k}\left( 1+O(q^{-1/2})    \right)
		\end{equation}
where the implied constant in the error term $O(q^{-1/2})$ depends only on $k$ and $g$.
	\end{Th}

\begin{remark}
	To establish Theorem \ref{theorem: main theorem}, we prove a result (Theorem \ref{theorem: main theorem for arbitrary partitions}) that is stronger than Theorem \ref{theorem: main theorem}. For any partition type of the set $\{1,2,\dots,k\}$, we provide an asymptotic count of rational functions $h\in I(f,E)$ whose associated principal divisor on $C\!\smallsetminus\! \text{supp}(E)$ has that partition type.
\end{remark}

\begin{remark}
	Note that since $I(f,E)$ is a short interval, all poles of $f$ lie in $\text{supp}(E)$, and for each $\mathfrak{p}\in\text{supp}(E)$, the order $\text{ord}_{\mathfrak{p}}(f)$ of each pole of $f$ at $\mathfrak{p}$ is strictly greater than the order $m_{\mathfrak{p}}$ of $E$ at $\mathfrak{p}$. Because $k$ is the sum of orders of all poles of $f$, Definition~\ref{definition: general SI} then implies that $k$ is equal to the sum $\sum_{\mathfrak{p}\in\text{supp}(E)}\text{max}\{m_{\mathfrak{p}},\text{ord}_{\mathfrak{p}}(f)\}$.
\end{remark}

\end{subsection}

%%%%%%%%%%%%%%%%%%%%%%%%%%%%%%%%%%%%%%%
	
\begin{subsection}{Outline of the paper}\label{subsection: outline of the paper}
	In broad outline, our strategy for proving Theorem \ref{theorem: main theorem} is similar to the strategy taken in \cite{BBR15}; the key insight of the present paper is that specific positivity hypotheses for divisors on $C$ allow one to adapt the steps of the original argument in \cite[\S 3 and \S 4]{BBR15} to a curve of arbitrary genus. In more detail, the outline of the paper is as follows.
	
	In \S\ref{sec: SI and Galois fields} we review the divisor theory and positivity conditions we will need. We introduce a variety parameterizing the elements of a short interval, and we use this variety to describe the generic element in a short interval. In \S\ref{section: Galois group of a generic element in a short interval} we explain how to associate a Galois group to the generic element. Most of the work in this section lies in showing that the Galois group is well defined. In \S\ref{section: calculation of the Galois group} we calculate the Galois group. Specifically, we show that it is isomorphic to a symmetric group by verifying the conditions in a particular characterization of the symmetric group. In \S\ref{section: counting argument} we use our knowledge of this Galois group, along with some basic facts about \'etale morphisms, to show that a key counting result in  \cite{BBR15}, originally stated only for the genus zero case, can be extended to a count in any genus. Finally, we use this count to prove Theorem \ref{theorem: main theorem} and its stronger form Theorem \ref{theorem: main theorem for arbitrary partitions}. Our arguments in \S\ref{section: counting argument} make crucial use of the Lang-Weil estimates \cite{LangWeil1954} and Bary-Soroker's Chebotarev-type result \cite[Proposition 2.2]{BarySoroker2012}.

\end{subsection}

%%%%%%%%%%%%%%%%%%%%%%%%%%%%%%%%%%%%%%%
	
\begin{subsection}{Acknowledgments}
	The authors would like to thank Lior Bary-Soroker and Michael Zieve for many conversations during our work on this paper that were crucial to its success. We also extend a warm thank you to Jeff Lagarias for comments on a draft of the paper and for his assistance in formulating prime density theorems in number fields. The exposition benefited greatly from an invitation to speak at the Palmetto Number Theory Series, held at the University of South Carolina and organized by Matthew Boylan, Michael Filaseta, and Frank Thorne, and from an invitation by Jordan Ellenberg to speak in the Number Theory Seminar at the University of Wisconsin\textendash Madison.
	
	We are especially thankful to Brian Conrad for looking closely at an earlier version of this paper, for pointing out a gap in our uniformity argument (now filled), for pointing out that we needed to prove what is now Proposition \ref{proposition: etale!}, and for suggesting that we use \cite[\S1, Lemma 1.5]{FK} to do so. 
	
	The authors conducted the research that lead to this paper while at the University of Michigan and while the second author was a visiting researcher at L'Institut des Hautes \'Etudes Scientifiques, at L'Institut Henri Poincar\'e, and at the Max Planck Institute for Mathematics. We thank all four institutions for their hospitality. The first author was partially supported by Michael Zieve's NSF grant DMS-1162181. Support for the second author came from NSF RTG grant DMS-0943832 and from Le Laboratoire d'Excellence CARMIN.

\end{subsection}
	
\end{section}

\vskip .5cm

%%%%%%%%%%%%%%%%%%%%%%%%%%%%%%%%%%%%%%
%%%%%%%%%%%%%%%%%%%%%%%%%%%%%%%%%%%%%%

\begin{section}{Short intervals on curves}\label{sec: SI and Galois fields}

	Fix a finite field $\FF_{\!q}$, an algebraic closure $\overline{\FF_{\!q}}\big/\FF_{\!q}$, and a smooth projective geometrically irreducible curve $C$ over $\FF_{\!q}$ of arithmetic genus $g$.

\begin{subsection}{Divisors on a curve}\label{subsec: Div on Curves}
	We make extensive use of the theory of divisors on algebraic varieties (see \cite[\S II.6]{Hartshorne} for instance). We briefly review the most pertinent aspects of the theory.
	
	By a {\em divisor on $C$}, we mean a Weil divisor on $C$.  We denote the support of a divisor $D$ by $\text{supp}(D)$, although we drop the distinction between $D$ and its support when it will not lead to confusion. For instance, we write $C\!\smallsetminus\! D$ instead of $C\!\smallsetminus\!\text{supp}(D)$. If $f$ is a rational function on $C$, we denote its associated principal divisor by $\text{div}(f)$. Given a divisor $D=\sum_{\mathfrak{p}\in C}m_{\mathfrak{p}}\ \mathfrak{p}$ on $C$, its {\em divisor of zeros} and {\em divisor of poles} are, respectively, the effective divisors
	$$
	D_{+}
	\ \ \Def\ 
	\sum_{m_{\mathfrak{p}}>0}m_{\mathfrak{p}}\ \mathfrak{p}
	\ \ \ \ \ \ \mbox{and}\ \ \ \ \ \ 
	D_{-}
	\ \ \Def\ 
	\sum_{m_{\mathfrak{p}}<0}-m_{\mathfrak{p}}\ \mathfrak{p}.
	$$
Note that $D=D_{+}-D_{-}$.
	
	Each divisor $D$ on $C$ determines a sheaf $\mathscr{O}(D)$ of rational functions on $C$ whose value at each open subset $U\subset C$ is
	$$
	\mathscr{O}(D)(U)
	\ \ \Def\ \ 
	\{0\}\sqcup\big\{f\in \FF_{\!q}(C)^{\times}:\Div(f)\big|_{U}\geq -D\big|_{U}\big\}.
	$$
For each $i\ge 0$, the $\FF_{\!q}$-vector space $H^{i}\big(C,\mathscr{O}(D)\big)$ is finite dimensional. We stress that according to the definition of $\mathscr{O}(D)$ that we use, the space of global sections $H^{0}\big(C,\mathscr{O}(D)\big)$ is canonically a space of rational functions on $C$.

	For each $m\ge 0$, fix homogeneous coordinates $x_{0},\dots,x_{m}$ on $\PP^{m}$. Let $V(x_{0})\subset\PP^{m}$ denote the hyperplane cut out by $x_{0}$. If $H^{0}\big(C,\mathscr{O}(D)\big)$ admits a basis $\{f_{0},\cdots,f_{m}\}$ such that at least one of the functions $f_{i}$ is non-vanishing at each $\mathfrak{p}\in C$, we say that $D$ is {\em basepoint free}. If $D$ is basepoint free, then our basis gives rise to a morphism
	\begin{equation}\label{eq: map from basepoint free complete linear system}
	\varphi=[f_{0}:\cdots:f_{m}]:C\longrightarrow \PP^{m}
	\end{equation}
into projective space of dimension
	$$
	m\ =\ \text{dim}_{\ \!}H^{0}\big(C,\mathscr{O}(D)\big)-1.
	$$
The divisor $D$ is {\em very ample} if $D$ is basepoint free and the morphism \eqref{eq: map from basepoint free complete linear system} is a closed embedding. Every divisor $D$ on $C$ satisfying $\text{deg}_{\ \!}D\ge 2g+1$ is very ample.

	If $E$ is an effective very ample divisor on $C$, then the basis $\{f_{0},\dots,f_{m}\}$ of $H^{0}\big(C,\mathscr{O}(E)\big)$ can be chosen so that $f_{0}=1$, with
	\begin{equation}\label{eq: special choice of basis}
	\varphi(\text{supp}(E))
	\ =\ 
	\varphi(C)\cap V(x_{0})
	\ \ \ \ \ \ \mbox{and}\ \ \ \ \ \ 
	C\!\smallsetminus\! E
	\ =\ 
	\varphi^{-1}\big(\PP^{m}\!\smallsetminus\! V(x_{0})\big).
	\end{equation}
In particular, if $E$ is an effective very ample divisor, then the open subscheme $C\!\smallsetminus\! E\subset C$ is affine, and its ring of regular functions is generated by the coordinates $x_{1},\dots,x_{m}$ on $\PP^{m}\!\smallsetminus\! V(x_{0})$. We consistently use the notation
	$$
	R
	\ \Def\ 
	\mbox{ring of regular functions on } C\!\smallsetminus\! E.
	$$
\begin{remark}
	Note that if $D_{0}$ and $D$ are divisors on $C$ satisfying $D_{0}\le D$, then we have a natural inclusion $H^{0}\big(C,\mathscr{O}(D_{0})\big)\subset H^{0}\big(C,\mathscr{O}(D)\big)$. Thus if $D_{0}$ is very ample and $D_{0}\le D$, then $D$ is also very ample.
\end{remark}
	
\begin{remark}
	Given a field extension $K/\FF_{\!q}$, each point $\mathfrak{p}$ in $C$ has a unique factorization $\mathfrak{p}=\mathfrak{q}^{e_{1}}_{1}\!\cdots\mathfrak{q}^{e_{n}}_{n}$ locally on $C_{K}=\text{Spec}(K)\!\times_{\text{Spec}(\FF_{\!q})}\!C$. The pullback of $E=\sum m_{\mathfrak{p}}\ \!\mathfrak{p}$ to $C_{K}$ is the divisor
	$$
	E_{K}
	\ \Def\ 
	\sum_{\mathfrak{p}\in C}
	\Big(\sum^{n}_{i=1}
	m_{\mathfrak{p}}\ \!e_{i}\ \mathfrak{q}_{i}
	\Big).
	$$
Note that $\deg{E}=\deg{E_{K}}$, and that if $E$ is effective, then $E_{K}$ is effective as well. The sheaf $\mathscr{O}(E)$ on $C$ pulls back to a sheaf $\mathscr{O}(E)_{K}$ on $C_{K}$, and we have a canonical isomorphism $\mathscr{O}(E_{K})\cong\mathscr{O}(E)_{K}$ \cite[\S9.4.2]{Harder}, thus $E_{K}$ is very ample whenever $E$ is.
\end{remark}
	
\end{subsection}

%%%%%%%%%%%%%%%%%%%%%%%%%%%%%%%%%%%%%%

\begin{subsection}{Generic element in a short interval}\label{subsection: generic element in a short interval}
	Let $E$ be an effective very ample divisor on $C$. Then following Definition \ref{definition: general SI}, each regular function $f$ on $C\!\smallsetminus\! E$ determines an interval
	$$
	I(f,E)
	\ =\ 
	f+H^{0}\big(C,\mathscr{O}(E)\big).
	$$
The fact that $f\in R$ and $H^{0}\big(C,\mathscr{O}(E)\big)\subset R$ implies that $I(f,E)\subset R$.
	
	Choose a basis $\{1,f_{1},\dots,f_{m}\}$ of $H^{0}\big(C,\mathscr{O}(E)\big)$ as in \S\ref{eq: special choice of basis}. Then we have a corresponding interpretation of
	$$
	\A^{\!m+1}
	\ \ \!\Def\ \ \!
	\text{Spec}_{\ \!}\FF_{\!q}[A_{0},\dots,A_{m}]
	\ =\ 
	\text{Spec}_{\ \!}\FF_{\!q}[\mathbf{A}]
	$$
as a variety parameterizing the functions in $I(f,E)$. Let $\FF_{\!q}(\mathbf{A})$ denote the field of rational functions $\FF_{\!q}(A_{0},\dots,A_{m})$, and define
	$$
	R[\mathbf{A}]
	\ \overset{{}_{\text{def}}}{=}\ 
	R\otimes_{\FF_{\!q}}\FF_{\!q}[\mathbf{A}]
	\ \ \ \ \ \ \ \ \ \ \mbox{and}\ \ \ \ \ \ \ \ \ 
	R(\mathbf{A})
	\ \overset{{}_{\text{def}}}{=}\ 
	R\otimes_{\FF_{\!q}}\FF_{\!q}(\mathbf{A}).
	$$
On the trivial family of curves $\A^{\!m+1}\times (C\!\smallsetminus\! E)=\sSpec{R[\mathbf{A}]}$, we have a regular function
	\begin{equation}\label{eq: formula for generic element of I}
	\mathcal{F}_{\!\mathbf{A}}
	\ \ \Def\ \ 
	f+A_{0}+\sum_{i=1}^{m}A_{i}f_{i}.
	\end{equation}
See Figure \ref{figure: main} for a depiction of the scheme $V(\mathcal{F}_{\!\mathbf{A}})$ cut out by $\mathcal{F}_{\!\mathbf{A}}$ inside the family $\mathbb{A}^{\!m+1}\!\times\! (C\!\smallsetminus\! E)$. The restriction of $\mathcal{F}_{\!\mathbf{A}}$ to the generic fiber $\sSpec{R(\mathbf{A})}$ of this trivial family $\A^{\!m+1}\times (C\!\smallsetminus\! E)$ describes the generic element of $I(f,E)$. If we denote $\FF_{\!q}$-rational points in $\A^{\!m+1}=\sSpec{\FF_{\!q}[\mathbf{A}]}$ as $(m+1)$-tuples $\mathbf{a}=(a_{0},\dots,a_{m})$, then for each $\mathbf{a}\in\A^{\!m+1}(\FF_{\!q})$, the restriction $\mathcal{F}_{\!\mathbf{a}}$ of \eqref{eq: formula for generic element of I} to the fiber $\{\mathbf{a}\}\times (C\!\smallsetminus\! E)\cong C\!\smallsetminus\! E$ is an element of $I(f,E)$.
	The value $\pi_{C}\big(I(f,E)\big)$ becomes the count of a particular set of $\FF_{\!q}$-rational points in $\A^{\!m+1}$:
	\begin{equation}\label{eq: prime count as point count}
	\pi_{C}\big(I(f,E)\big)
	\ =\ 
	\#\big\{\mathbf{a}\in\A^{\!m+1}(\FF_{\!q}):\mbox{the ideal }(\mathcal{F}_{\!\mathbf{a}})\subset R\mbox{ is prime}\big\}.
	\end{equation}

\begin{figure}
	$$
	\scalebox{.9}{$
	\begin{xy}
	(0,0)*+{
	\scalebox{.75}{
	$
	\includegraphics[scale=.17]{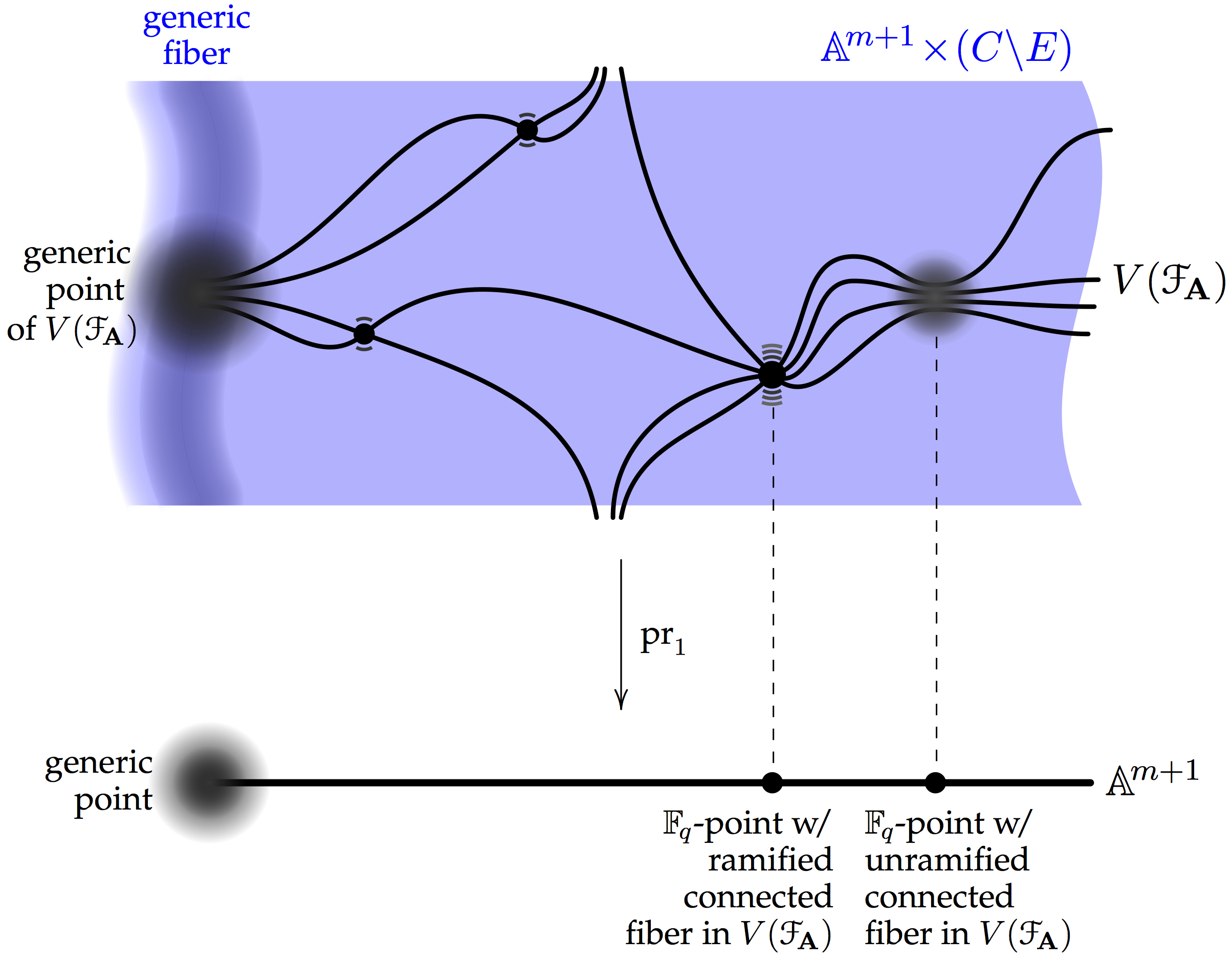}
	$}
	};
	\end{xy}
	$}
	$$
	
\caption{The scheme $V(\mathcal{F}_{\!\mathbf{A}})$ inside $\mathbb{A}^{\!m+1}\!\times\! (C\!\smallsetminus\! E)$. By Lemmas \ref{lemma: primality of (F_A)} and \ref{lemma: separability of (F_A)}, $V(\mathcal{F}_{\!\mathbf{A}})$ intersects the generic fiber at a single unramified point. Points $\mathbf{a}\in\A^{\!m+1}(\FF_{\!q})$ with unramified connected fiber in $V(\mathcal{F}_{\!\mathbf{A}})$ describe elements $\mathcal{F}_{\!\mathbf{a}}\in I(f,E)$ counted in $\pi_{C}\big(I(f,E)\big)$.}
\label{figure: main}
\end{figure}
	
\end{subsection}

\end{section}

\vskip .5cm

%%%%%%%%%%%%%%%%%%%%%%%%%%%%%%%%%%%%%%
%%%%%%%%%%%%%%%%%%%%%%%%%%%%%%%%%%%%%%

\begin{section}{Galois group of a generic element in a short interval}\label{section: Galois group of a generic element in a short interval}
	
	For any field $K$ and any irreducible separable polynomial $f\in K[t]$, the residue field $\kappa(f)\overset{{}_{\text{def}}}{=}K[t]/(f)$ admits a unique splitting field $\text{split}(f)$ inside any separable closure $\overline{K}=\overline{\kappa(f)}$. When we interpret $\kappa(f)$ as the field obtained by adjoining a single root of $f$ to $K$, it becomes natural to construct $\text{split}(f)$ as the field obtained by adjoining all roots of $f$ to $K$. We can also construct $\text{split}(f)$ without any explicit reference to roots of $f$. Indeed, $\text{split}(f)$ is the normal closure of $\kappa(f)$ inside $\overline{K}$ \cite[Theorem~2.9.5.(4)]{Roman}. This latter characterization of the splitting field generalizes to the higher genus setting and, as we demonstrate in the present section, allows us to define the Galois group of the generic element in short intervals on $C$.

%%%%%%%%%%%%%%%%%%%%%%%%%%%%%%%%%%%%%%

\begin{subsection}{The setting of \S\ref{section: Galois group of a generic element in a short interval} and \S\ref{section: calculation of the Galois group}}\label{subsection: setting}
	The following datum is to remain fixed throughout \S\ref{section: Galois group of a generic element in a short interval} and \S\ref{section: calculation of the Galois group}: Let $E$ be an effective very ample divisor on $C$, define $R$ to be the ring of regular functions on the affine curve $C\!\smallsetminus\! E$, and let $f\in R$ be a regular function on $C\!\smallsetminus\! E$ satisfying
	\begin{equation}\label{eq: important inequality at E}
	-\nu_{\mathfrak{p}}(f)
	\ >\ 
	\nu_{\mathfrak{p}}(E)
	\ \ \ \ \ \ \mbox{for all}\ \ \ \mathfrak{p}\in\text{supp}(E),
	\end{equation}
where $\nu_{\mathfrak{p}}(E)$ denotes the coefficient of $\mathfrak{p}$ in $E$. Define
	$$
	k
	\ \Def\ 
	\text{deg}\big(\text{div}(f)_{-}\big).
	$$
Note that the inequality \eqref{eq: important inequality at E} and the quantity $k$ are unaffected by base change along any field extension $K/\FF_{\!q}$.
Let $I(f,E)$ be the short interval defined by $f$ and $E$. Let $\mathcal{F}_{\!\mathbf{A}}$ be the generic element in $I(f,E)$ as defined in \eqref{eq: formula for generic element of I}. Fix an algebraic closure  $\overline{\FF_{\!q}(\mathbf{A})}$ such that $\overline{\FF_{\!q}}\subset\overline{\FF_{\!q}(\mathbf{A})}$.
	
\begin{remark}
	In the case where $g=0$ and $E$ is an effective divisor on $\PP^{1}$ supported at $\infty$, we have $H^{0}\big(\PP^{1},\mathscr{O}(E)\big)=\FF_{\!q}[t]^{\le m}$, where $m=\text{deg}_{\ \!}E$. The choice of a regular function $f$ amounts to the choice of a polynomial $f\in \FF_{\!q}[t]$, and $k=\text{deg}\big(\text{div}(f)_{-}\big)=\text{deg}_{\ \!}f$. Thus the inequality \eqref{eq: important inequality at E} reduces to the requirement $m<k$ that appears in the form ``$\varepsilon_{0}<1$" in Theorem \ref{theorem: main from BBR}.
\end{remark}
	
\begin{remark}
	In \S\ref{section: counting argument}, where we consider the asymptotic behavior of $I(f,E)$, we will allow $E$ and $f$ to vary subject to the constraint \eqref{eq: important inequality at E}.
\end{remark}

\end{subsection}

%%%%%%%%%%%%%%%%%%%%%%%%%%%%%%%%%%%%%%

\begin{subsection}{The splitting field and Galois group of a relative separable point}\label{subsection: splitting field of a point in a curve}	
	We can associate Galois groups to a large class of points in $C\!\smallsetminus\! E$ as follows:

\begin{definition}\label{definition: splitting field of a point}
\normalfont
	Let $K/\FF_{\!q}(\mathbf{A})$ be an algebraic extension. For a prime ideal $\mathfrak{P}$ in the ring $K\otimes_{\FF_{\!q}(\mathbf{A})}R(\mathbf{A})$, denote by $\kappa(\mathfrak{P})$ the residue field of $\mathfrak{P}$. The {\em splitting field of $\mathfrak{P}$} ({\em over $K$}), denoted $\text{split}(\mathfrak{P})$ or $\text{split}(\mathfrak{P}/K)$, is the normal closure of $\kappa(\mathfrak{P})$ in $\overline{\FF_{\!q}(\mathbf{A})}$.
	
	If the extension $\kappa(\mathfrak{P})/K$ is separable, then the {\em Galois group of $\mathfrak{P}$} is
	\begin{equation*}
	\text{Gal}\big(\mathfrak{P}\big/K\big)
	\ \Def\ 
	\text{Gal}\big(\text{split}(\mathfrak{P})/K\big).
	\end{equation*}
\end{definition}
	
\begin{remark}\label{remark: Galois action on split prime}
	For a prime ideal $\mathfrak{P}$ in $K\otimes_{\FF_{\!q}(\mathbf{A})}R(\mathbf{A})$, the fact that $\kappa(\mathfrak{P})/K$ is separable is equivalent to the statement that
	\begin{equation}\label{eq: prime splits over its splitting field}
	\text{split}(\mathfrak{P})\underset{K}{\otimes}\kappa(\mathfrak{P})
	\ \ \ \cong\ \ 
	\prod_{i=1}^{\text{deg}_{\ \!}\mathfrak{P}}\!\!
	\text{split}(\mathfrak{P})
	\end{equation}
(see \cite[Proposition 5.3.9, Definition 5.3.12 and Proposition 5.3.16.(1)]{Szamuely}).
	Since $K\otimes_{\FF_{\!q}}R(\mathbf{A})$ is a Dedekind domain, the isomorphism \eqref{eq: prime splits over its splitting field} is equivalent to the statement that in the ring $\text{split}(\mathfrak{P})\otimes_{\FF_{\!q}(\mathbf{A})}R(\mathbf{A})$, the ideal $\text{split}(\mathfrak{P})\otimes_{K}\mathfrak{P}$ has prime factorization
	\begin{equation}\label{eq: prime splitting as factorization}
	\text{split}(\mathfrak{P})\otimes_{K}\mathfrak{P}
	\ \ =\ \ 
	\mathfrak{Q}_{1}\cdots\mathfrak{Q}_{\deg{\mathfrak{P}}},
	\end{equation}
	where $\deg \mathfrak{Q}_{i}=1$ and $\kappa(\mathfrak{Q}_{i})\cong\text{split}(\mathfrak{P})$ for each $1\le i\le \deg \mathfrak{P}$. The Galois group $\text{Gal}(\mathfrak{P}/K)$ acts faithfully and transitively on the prime factors $\mathfrak{Q}_{i}$.		
\end{remark}
	
\end{subsection}

%%%%%%%%%%%%%%%%%%%%%%%%%%%%%%%%%%%%%

\begin{subsection}{Primality and separability of the generic element}\label{subsection: primality and separability of the generic element}
	
	For any field extension $K/\FF_{\!q}$, define
	$$
	\begin{array}{rcl}
	R_{K}[\mathbf{A}]\ \Def\ K\otimes_{\FF_{\!q}}R[\mathbf{A}]
	\ \ \ \ \ \ \ \ \ \mbox{and}\ \ \ \ \ \ \ \ \ 
	R_{K}(\mathbf{A})\ \Def\ R\otimes_{\FF_{\!q}}K(\mathbf{A}).
	\end{array}
	$$
The canonical morphism $R[\mathbf{A}]\longrightarrow R_{K}[\mathbf{A}]$ lets us interpret both $f$ and $\mathcal{F}_{\!\mathbf{A}}$ as elements of $R_{K}[\mathbf{A}]$. By \S\ref{subsec: Div on Curves} and \cite[Corollaire 6.9.9]{EGAIII2}, we have
	$$
	H^{0}\big(C_{K},\mathscr{O}(E_{K})\big)
	\ \cong\ 
	K\!\otimes_{\FF_{\!q}}\!H^{0}\big(C,\mathscr{O}(E)\big),
	$$
and $\A^{\!m+1}_{K}=\sSpec{K[\mathbf{A}]}$ becomes a variety parameterizing elements in $I(f,E_{K})$.
	
\begin{lemma}\label{lemma: primality of (F_A)}
	\normalfont
	For any field extension $K/\FF_{\!q}$, the ideal $(\mathcal{F}_{\!\mathbf{A}})\subset R_{K}(\mathbf{A})$ generated by $\mathcal{F}_{\!\mathbf{A}}$ is prime in $R_{K}(\mathbf{A})$.
\end{lemma}
\begin{proof}
	Let $V(\mathcal{F}_{\!\mathbf{A}})$ be the variety cut out by $\mathcal{F}_{\!\mathbf{A}}$ in $\A^{\!m+1}\times\big(C_{K}\!\smallsetminus\! E_{K}\big)$ . The projections
	\begin{equation}\label{eq: morphism 1 for primality proof}
	\begin{aligned}
	\begin{xy}
	(0,0)*+{\A^{\!m+1}\!\times\!\big(C_{K}\!\smallsetminus\! E_{K}\big)}="1";
	(-30,0)*+{\A^{\!m+1}}="2";
	(30,0)*+{C_{K}\!\smallsetminus\! E_{K}}="3";
	{\ar_{\text{pr}_{1}\ \ \ \ \ \ \ \ \ } "1"; "2"};
	{\ar^{\ \ \ \ \ \ \ \ \ \text{pr}_{2}} "1"; "3"};
	\end{xy}
	\end{aligned}
	\end{equation}
restrict to morphisms
	$$
	\begin{xy}
	(0,0)*+{V(\mathcal{F}_{\!\mathbf{A}})}="1";
	(-30,0)*+{\A^{\!m+1}}="2";
	(30,0)*+{C_{K}\!\smallsetminus\! E_{K}.}="3";
	{\ar_{\text{pr}_{1}|_{V(\mathcal{F}_{\!\mathbf{A}})}} "1"; "2"};
	{\ar^{\text{pr}_{2}|_{V(\mathcal{F}_{\!\mathbf{A}})}} "1"; "3"};
	\end{xy}
	$$
Assume that $(\mathcal{F}_{\!\mathbf{A}})\subset R_{K}(\mathbf{A})$ is not prime. Then either the morphism $\text{pr}_{1}|_{V(\mathcal{F}_{\!\mathbf{A}})}$ has empty generic fiber, or else the subscheme $V(\mathcal{F}_{\!\mathbf{A}})\subset\A^{\!m+1}_{K}\!\times\!\big(C_{K}\!\smallsetminus\! E_{K}\big)$ has more than one irreducible component.
	
	Comparing the strict inequalities \eqref{eq: important inequality at E} with the inequality defining the inclusion
	$$
	A_{0}+\sum_{i=1}^{m} A_{i}f_{i}
	\ \ \in\ \ 
	H^{0}\big(C_{K(\mathbf{A})},\mathscr{O}(E_{K(\mathbf{A})})\big),
	$$
we see that $\nu_{\mathfrak{p}}(\mathcal{F}_{\!\mathbf{A}})< -m_{\mathfrak{p}}$ for all $\mathfrak{p}\in\text{supp}_{\ \!}E_{K(\mathbf{A})}$. Thus
	\begin{equation}\label{equation: important inequality for zeros of F_A}
	\text{div}(\mathcal{F}_{\!\mathbf{A}})_{+}\ne 0
	\ \ \ \ \ \ \mbox{and}\ \ \ \ \ \ 
	\text{supp}\big(\text{div}(\mathcal{F}_{\!\mathbf{A}})_{+}\big)\ \subset\ C_{K}\!\smallsetminus\! E_{K}.
	\end{equation}
In particular, $\mathcal{F}_{\!\mathbf{A}}$ is not a unit in $R_{K}(\mathbf{A})$, and $\text{pr}_{1}|_{V(\mathcal{F}_{\!\mathbf{A}})}$ does not have empty generic fiber.
	
	Because $V(\mathcal{F}_{\!\mathbf{A}})$ is pure of codimension-$1$ inside $\A^{\!m+1}\!\times \big(C_{K}\!\smallsetminus\! E_{K}\big)$, whereas $C_{K}\!\smallsetminus\! E_{K}$ is $1$-dimensional, an irreducible component of $V(\mathcal{F}_{\!\mathbf{A}})$ is either a whole fiber of the projection $\text{pr}_{2}$ in \eqref{eq: morphism 1 for primality proof} over a closed point of $C_{K}\!\smallsetminus\! E_{K}$, or else its generic point lies over the generic point of $C_{K}\!\smallsetminus\! E_{K}$. For any point $x\in C_{K}\!\smallsetminus\! E_{K}$, the function $\mathcal{F}_{\!\mathbf{A}}|_{x}=f(x)+A_{0}+\sum A_{i}f_{i}(x)\in\kappa(x)[\mathbf{A}]$ is linear in the variables $A_{i}$, and is nonzero since $A_{0}$ has coefficient $1$. For closed points $x\in C_{K}\!\smallsetminus\! E_{K}$, this shows that closed fibers of the morphism $\text{pr}_{2}$ in \eqref{eq: morphism 1 for primality proof} cannot be irreducible components of $V(\mathcal{F}_{\!\mathbf{A}})$. Over the generic point $\xi$ of $C_{K}\!\smallsetminus\! E_{K}$, linearity of the nonzero function $\mathcal{F}_{\!\mathbf{A}}|_{\xi}$ implies that the ideal $(\mathcal{F}_{\!\mathbf{A}}|_{\xi})\subset K(C_{K})[\mathbf{A}]$ is prime. Thus $V(\mathcal{F}_{\!\mathbf{A}})$ has a unique irreducible component.
\end{proof}

\begin{remark}\label{remark: point associated to F_A}
	Since $C$ is a curve, Lemma \ref{lemma: primality of (F_A)} implies that the subscheme $V(\mathcal{F}_{\!\mathbf{A}})_{K}\subset C_{K(\mathbf{A})}$ consists of a single closed point $\mathfrak{P}_{K}$. The residue field $\kappa(\mathfrak{P}_{K})=R_{K}(\mathbf{A})\big/(\mathcal{F}_{\!\mathbf{A}})$ is a finite extension of $K(\mathbf{A})$.
\end{remark}
	
\begin{lemma}\label{lemma: separability of (F_A)}
	\normalfont
	For any field extension $K/\FF_{\!q}$, the extension $\kappa(\mathfrak{P}_{K})\big/K(\mathbf{A})$ in Remark \ref{remark: point associated to F_A} is separable.
\end{lemma}
\begin{proof}
	For each homogenous function $h$ on $\PP^{m}$, let $D(h)$ denote the distinguished open subscheme of $\PP^{m}$ where $h$ is nonzero. The fact that $E$ is very ample allows us to choose a polynomial $f(\mathbf{x})\in \FF_{\!q}[x_{1},\dots,x_{m}]$ such that the regular function $f\in R$ is the restriction of $f(\mathbf{x})$ to $C\!\smallsetminus\! E=C\cap D(x_{0})$. The function $\mathcal{F}_{\!\mathbf{A}}$ on $C_{K(\mathbf{A})}$ is then the restriction of the function
	$$
	\mathcal{F}_{\!\mathbf{A}}(\mathbf{x})
	\ \ \Def\ \ 
	f(\mathbf{x})+A_{0}+\sum_{i=1}^{m}A_{i}\ x_{i}
	\ \ \ \ \mbox{defined on}\ \ \ \ 
	D(x_{0}).
	$$
The curve $C$ is smooth, therefore there exist functions $y\in \FF_{\!q}[x_{1},\dots,x]$ and $r_{1},\dots,r_{m-1}\in \FF_{\!q}[x_{1},\dots,x_{m},\frac{1}{y}]$ for which the point $\mathfrak{P}_{K}$ in Remark \ref{remark: point associated to F_A} lies in the affine open neighborhood
	$$
	C_{K(\mathbf{A})}\cap D(x_{0}y)_{K(\mathbf{A})}
	\ \cong\ 
	\sSpec{K(\mathbf{A})\big[x_{1},\dots,x_{m},\tfrac{1}{y}\big]}\big/(r_{1},\dots,r_{m-1}),
	$$
and such that the determinant of the $(m-1)\times(m-1)$-minor $M_{mm}$ in the matrix
	\begin{equation}\label{equation: key matrix 1}
	\!\!\!\!\!
	M
	\ \ \Def\ 
	\begin{xy}
	(0,0)*+{
	\left(\!\!
	\begin{array}{ccccc}
	\frac{\partial r_{1}}{\partial x_{1}}
	&
	\frac{\partial r_{1}}{\partial x_{2}}
	&
	\cdots
	&
	\frac{\partial r_{1}}{\partial x_{m-1}}
	&
	\frac{\partial r_{1}}{\partial x_{m}}
	\\[6pt]
	\frac{\partial r_{2}}{\partial x_{1}}
	&
	\frac{\partial r_{2}}{\partial x_{2}}
	&
	\cdots
	&
	\frac{\partial r_{2}}{\partial x_{m-1}}
	&
	\frac{\partial r_{2}}{\partial x_{m}}
	\\[6pt]
	\vdots
	&
	\vdots
	&
	\ddots
	&
	\vdots
	&
	\vdots
	\\[6pt]
	\frac{\partial r_{m-1}}{\partial x_{1}}
	&
	\frac{\partial r_{m-1}}{\partial x_{2}}
	&
	\cdots
	&
	\frac{\partial r_{m-1}}{\partial x_{m-1}}
	&
	\frac{\partial r_{m-1}}{\partial x_{m}}
	\\[6pt]
	& & & &\ 
	\\[-8pt]
	\!\!\frac{\partial \mathcal{F}_{\!\mathbf{A}}(\mathbf{x})}{\partial x_{1}}\!\!\!
	&
	\!\!\frac{\partial \mathcal{F}_{\!\mathbf{A}}(\mathbf{x})}{\partial x_{2}}\!\!\!
	&
	\cdots
	&
	\!\!\frac{\partial \mathcal{F}_{\!\mathbf{A}}(\mathbf{x})}{\partial x_{m-1}}\!\!\!
	&
	\!\!\frac{\partial \mathcal{F}_{\!\mathbf{A}}(\mathbf{x})}{\partial x_{m}}\!\!\!
	\end{array}
	\!\!\right)
	\ \ \ \ \ \ \ \ \ \!};
	(-10.5,4.7)*+{
	\begin{tikzpicture}
	\draw[gray, very thick] (0,0) -- (4.6,0) -- (4.6,3) -- (0,3) -- cycle;
	\end{tikzpicture}
	};
	(-10,22)*+{\mbox{{\color{gray} minor $M_{mm}$}}};
	\end{xy}
	\end{equation}
is invertible on $C_{K(\mathbf{A})}\cap D(x_{0}y)_{K(\mathbf{A})}$. The entries in the last row of $M$ all have the explicit form
	$$
	\frac{\partial\mathcal{F}_{\!\mathbf{A}}(\mathbf{x})}{\partial x_{i}}
	\ =\ 
	\frac{\partial f(\mathbf{x})}{\partial x_{i}}+A_{i}.
	$$
Hence, the term associated to $M_{mm}$ in the cofactor expansion of $\text{det}(M)$ is the only cofactor term in which $A_{m}$ appears. The coefficient of $A_{m}$ in this term is nonzero at $\mathfrak{P}_{K}$, therefore $\text{det}(M)$ is nonzero at $\mathfrak{P}_{K}$. The $K(\mathbf{A})$-scheme $\text{Spec}_{\ \!}\kappa(\mathfrak{P}_{K})$ is then smooth of relative dimension $0$ over $K(\mathbf{A})$, or equivalently, $\text{Spec}_{\ \!}\kappa(\mathfrak{P}_{K})$ is \'etale over $\text{Spec}_{\ \!}K(\mathbf{A})$ \cite[\S I.3, Corollary 3.16]{Milne:80}, and the field extension $\kappa(\mathfrak{P}_{K})/K(\mathbf{A})$ is separable \cite[\S I.3, Proposition 3.2.(a) \& (e)]{Milne:80}.
\end{proof}

	From Lemmas \ref{lemma: primality of (F_A)} and \ref{lemma: separability of (F_A)}, we immediately have the following:

\begin{corollary}
\normalfont
	For each algebraic extension $K/\FF_{\!q}$, the prime ideal $\mathfrak{P}_{K}=(\mathcal{F}_{\!\mathbf{A}})\subset R_{K}(\mathbf{A})$ has an associated Galois group, which we henceforth denote
	\begin{center}
	\hfill
	$
	\text{Gal}\big(\mathcal{F}_{\!\mathbf{A}},K(\mathbf{A})\big)
	\ \ \Def\ \ 
	\text{Gal}\big(\mathfrak{P}\big/K(\mathbf{A})\big).
	$
	\hfill$\square$
	\end{center}
\end{corollary}
	
\begin{lemma}\label{lemma: inclusion of Galois groups under base change}
\normalfont
	For each algebraic field extension $K/\FF_{\!q}$, there is an inclusion of Galois groups
	\begin{equation}\label{lemma: inclusion for comparing Galois groups}
	\text{Gal}\big(\mathcal{F}_{\!\mathbf{A}},K(\mathbf{A})\big)
	\mono
	\text{Gal}\big(\mathcal{F}_{\!\mathbf{A}},\FF_{\!q}(\mathbf{A})\big).
	\end{equation}
\end{lemma}
\begin{proof}
	Since $\kappa(\mathfrak{P}_{K})$ is isomorphic to the compositum $K\!\cdot\kappa(\mathfrak{P})\subset\overline{\FF_{\!q}(\mathbf{A})}$, we have isomorphisms
	\begin{equation}\label{equation: string of isomorphisms}
	\text{Gal}\big(\mathcal{F}_{\!\mathbf{A}},K(\mathbf{A})\big)
	\ \xrightarrow{\ \sim\ }\ 
	\text{Gal}\big(K\!\cdot\!\kappa(\mathfrak{P})\big/K(\mathbf{A})\big)
	\ \xrightarrow{\ \sim\ }\ 
	\text{Gal}\big(\kappa(\mathfrak{P}_{K})\big/K(\mathbf{A})\big).
	\end{equation}
Post-composing \eqref{equation: string of isomorphisms} with the inclusion $\text{Gal}\big(\kappa(\mathfrak{P}_{K})\big/K(\mathbf{A})\big)\hookrightarrow\text{Gal}\big(\kappa(\mathfrak{P}_{K})\big/\FF_{\!q}(\mathbf{A})\big)$, we obtain the embedding \eqref{lemma: inclusion for comparing Galois groups}.
\end{proof}

\begin{proposition}\label{proposition: etale!}
\normalfont
	The branch locus $Z\subset\A^{\!m+1}_{\FF_{\!q}}$ of the morphism $\text{pr}_{1}:V(\mathcal{F}_{\!\bold{A}})\lra\A^{\!m+1}$ has codimension $\ge 1$ in $\A^{\!m+1}_{\FF_{\!q}}$, and its compliment $\A^{\!m+1}\smallsetminus Z$ is the maximal open subset of $\A^{\!m+1}_{\FF_{\!q}}$ over which $V(\mathcal{F}_{\!\bold{A}})$ is finite \'etale. 
\end{proposition}
\begin{proof}
Lemma~\ref{lemma: separability of (F_A)} implies that $V(\mathcal{F}_{\!\bold{A}})$ is generically unramified over $\A^{m+1}_{\FF_q}$, and thus that $Z$ has codimension $\ge 1$ in $\A^{\!m+1}_{\FF_{\!q}}$. Define
		$$
		X
		\ \overset{{}_{\text{def}}}{=}\ 
		\A^{\!m+1}_{\FF_{\!q}}\smallsetminus Z
		\ \ \ \ \ \ \ \ \ \text{and}\ \ \ \ \ \ \ \ \ 
		Y
		\ \overset{{}_{\text{def}}}{=}\ 
		V(\mathcal{F}_{\!A})_{\A^{\!m+1}_{\FF_{\!q}}\smallsetminus Z}
		\ =\ 
		V(\mathcal{F}_{\!A})_{X}.
		$$
Then the resulting morphism $\text{pr}_{1}|_{Y}:Y\lra X$ is finite, surjective, and unramified of degree $k$. The variety $X$ is a normal, and surjectivity of $\text{pr}_{1}|_{Y}$ implies that for each $y\in Y$, the morphism of stalks $\mathscr{O}_{\!X,\text{pr}_{1\!}(y)}\lra\mathscr{O}_{Y,y}$ is injective \cite[\href{http://stacks.math.columbia.edu/tag/0CC1}{Tag 0CC1}, (1) \& (6)]{STACKS}. Thus by \cite[\S1, Lemma 1.5]{FK}, the morphism $\text{pr}_{1}|_{Y}$ is \'etale. Because $\A^{\!m+1}_{\FF_{\!q}}\!\smallsetminus\!X$ is the branch locus $Z$, this implies that $X$ is the maximal open subset of $\A^{\!m+1}_{\FF_{\!q}}$ over which $V(\mathcal{F}_{\!\bold{A}})$ is finite \'etale.
\end{proof}

\end{subsection}

\end{section}

\vskip .5cm

%%%%%%%%%%%%%%%%%%%%%%%%%%%%%%%%%%%%%%
%%%%%%%%%%%%%%%%%%%%%%%%%%%%%%%%%%%%%%

\begin{section}{Calculation of the Galois group}\label{section: calculation of the Galois group}

\begin{subsection}{A characterization of the symmetric group}
	Recall from \S\ref{subsection: setting} that we fix an effective very ample divisor $E$ on $C$ and a function $f$ regular on $C\!\smallsetminus\! E$ with poles satisfying the inequalities \eqref{eq: important inequality at E}, and that $k\overset{{}_{\text{def}}}{=}\text{deg}(\text{div}(f)_{-})$. Let $S_{k}$ denote the symmetric group on $k$ letters. Our goal in the present section is to prove the following:
	
\begin{theorem}\label{theorem: Galois group is S_k}
\normalfont
	Assume that $E$ satisfies one of the following two conditions
	\begin{itemize}
	\item[{\bf (a)}]
	There exists a very ample effective divisor $E_{0}$ on $C$ such that $E\ge 3E_{0}$;
	\item[{\bf (b)}]\vskip .1cm
	There exists a very ample effective divisor $E_{0}$ on $C$ such that $E\ge 2E_{0}$, $\text{char}_{\ \!}\FF_{\!q}=2$, and $df|_{C\!\smallsetminus\! E}$ vanishes on a finite nonempty set.\vskip .1cm
	\end{itemize}
Then the Galois group $\text{Gal}\big(\mathcal{F}_{\!\mathbf{A}},\FF_{\!q}(\mathbf{A})\big)$ is isomorphic to $S_{k}$.
\end{theorem}
	
\begin{remark}
	To prove Theorem \ref{theorem: Galois group is S_k}, we use the following characterization of $S_{k}$:
\end{remark}

\begin{lemma}\label{lemma: serre S_k}\normalfont
{\bf \cite[Lemma~4.4.3]{Serre2007Topics}.}
	A subgroup $G\subset S_k$ is equal to $S_{k}$ if and only if $G$ satisfies the following three conditions:
	\begin{itemize}
	\item[{\bf (i)}] $G$ is transitive;
	\item[{\bf (ii)}]\vskip .1cm $G$ is doubly transitive;
	\item[{\bf (iii)}]\vskip .1cm $G$ contains a transposition.\hfill$\square$
	\end{itemize}\vskip .2cm
\end{lemma}

\noindent
{\em Beginning of the proof of Theorem \ref{theorem: Galois group is S_k}.}	
	Observe that for any algebraic extension $K/\FF_{\!q}$ the condition \eqref{eq: important inequality at E} and its consequence \eqref{equation: important inequality for zeros of F_A}, combined with the fact that the total degree of any principal divisor is $0$, imply that $\text{deg}_{\ \!}\mathfrak{P}_{K}=k$. Thus by Remark \ref{remark: Galois action on split prime}, the Galois group $\text{Gal}\big(\mathcal{F}_{\!\mathbf{A}},K(\mathbf{A})\big)$ comes with a natural faithful action on a set of $k$ elements, namely the prime factors in \eqref{eq: prime splitting as factorization}. In this way, we obtain an embedding
	\begin{equation}\label{equation: inclusion from action on factors}
	\text{Gal}\big(\mathcal{F}_{\!\mathbf{A}},K(\mathbf{A})\big)\mono S_{k}
	\end{equation}
for each algebraic extension $K/\FF_{\!q}$. For the special case $K=\overline{\FF_{\!q}}$, Lemma \ref{lemma: inclusion of Galois groups under base change} tells us that the inclusion \eqref{equation: inclusion from action on factors} factors as
	$$
	\begin{xy}
	(0,0)*+{
	\ \Gal(\mathcal{F}_{\!\mathbf{A}}, \overline{\FF_q}(\mathbf{A}))\ 
	}="1";
	(30,10)*+{
	\Gal(\mathcal{F}_{\!\mathbf{A}}, \FF_q(\mathbf{A}))
	}="2";
	(50,0)*+{
	S_k.
	}="3";
	{\ar@{^{(}->} "1"; "2"};
	{\ar@{_{(}->} "2"; "3"};
	{\ar@{_{(}->} "1"; "3"};
	\end{xy}
	$$
It therefore suffices to check that the Galois group $\text{Gal}\big(\mathcal{F}_{\!\mathbf{A}},\overline{\FF_{\!q}}(\mathbf{A})\big)$ satisfies the three conditions in Lemma \ref{lemma: serre S_k}. We verify these conditions in \S\ref{subsection: transitivity and double transitivity} and \S\ref{subsection: presence of a transposition} below.

\end{subsection}

%%%%%%%%%%%%%%%%%%%%%%%%%%%%%%%%%%%%%%

\begin{subsection}{Transitivity and double transitivity}\label{subsection: transitivity and double transitivity}
	By Remark \ref{remark: Galois action on split prime}, the embedding \eqref{equation: inclusion from action on factors} realizes the group $\text{Gal}\big(\mathcal{F}_{\!\mathbf{A}}\big/\overline{\FF_{\!q}}(\mathbf{A})\big)$ as a transitive subgroup of $S_{k}$. Varifying condition (ii) of Lemma \ref{lemma: serre S_k} in the setting of Theorem \ref{theorem: Galois group is S_k} amounts to proving the following:
	
\begin{proposition}\label{proposition: doubly transitive on points}
	\normalfont
	If $E$ satisfies either of the conditions (a) or (b) in Theorem \ref{theorem: Galois group is S_k}, then the subgroup $\text{Gal}\big(\mathcal{F}_{\!\mathbf{A}}\big/\overline{\FF_{\!q}}(\mathbf{A})\big)\subset S_{k}$ is doubly transitive.
\end{proposition}

\begin{remark}\label{remark: regarding the star condition}
	Note that each of the conditions (a) and (b) of Theorem \ref{theorem: Galois group is S_k} imply the following weaker condition: for any degree-$1$ point $\mathfrak{q}$ in the support of $E$, the divisor $E-\mathfrak{q}$ is again effective and very ample.
\end{remark}

\begin{proof}[{\it Proof of Proposition \ref{proposition: doubly transitive on points}}.]
	Because $\text{Gal}\big(\mathcal{F}_{\!\mathbf{A}}\big/\overline{\FF_{\!q}}(\mathbf{A})\big)$ is transitive, it is enough to show that there exists a factor $\mathfrak{Q}_{i}$ in \eqref{eq: prime splitting as factorization} for which the stabilizer subgroup of $\mathfrak{Q}_{i}$ inside $\text{Gal}\big(\mathcal{F}_{\!\mathbf{A}}\big/\overline{\FF_{\!q}}(\mathbf{A})\big)$ is transitive on the set of factors $\{\mathfrak{Q}_{j}\}_{j\ne i}$.
	
	Fix a single $\overline{\FF_{\!q}}$-valued point $\mathfrak{q}\in V(f)\subset C_{\lilF}\!\smallsetminus\! E_{\lilF}$. Choose a hyperplane $L\subset\PP^{m}_{\lilF}$ such that the only point of $V(f)$ in $C_{\lilF}\cap L$ is $\mathfrak{q}$. Define $E'$ to be the effective divisor associated to the weighted intersection $C_{\lilF}\cap L$. Choose a linear form $\ell\in\overline{\FF_{\!q}}[x_{0},\dots,x_{m}]$ satisfying $E'=\text{div}(\ell)+E_{\lilF}$, and let $\mathfrak{h}\in\A^{\!m+1}_{\lilF}$ be the generic point of the hyperplane cut out by the equation $\sum_{i=0}^{m}A_{i}\ \!\ell_{i}=0$. Let $\mathcal{F}_{\!\mathfrak{h}}$ denote the restriction of $\mathcal{F}_{\!\mathbf{A}}$ to $\sSpec{R_{\kappa(\mathfrak{h})}}$. Then $\mathcal{F}_{\!\mathfrak{h}}$ factors as
	$$
	\mathcal{F}_{\!\mathfrak{h}}
	\ \ =\ \ 
	\ell\ \Big(\ \!f'+A'_{0}+\sum_{i=1}^{m-1}A'_{i}\ f'_{i}\ \Big)
	$$
where $\{1,f'_{1},\dots,f'_{m-1}\}$ is a basis of the subspace of $H^{0}\big(C_{\lilF},\mathscr{O}(E_{\lilF})\big)$ corresponding the hyperplane $V(\mathfrak{h})\subset \A^{\!m+1}_{\lilF}$, and where $f'$ is a regular function on $C_{\lilF}\!\smallsetminus\!(E'-\mathfrak{q})$ satisfying
	\begin{equation}\label{equation: new inequality}
	\text{div}(f')_{-}\ >\ E'-\mathfrak{q}.
	\end{equation}
The linear equivalence $E'\sim E_{\lilF}$ makes $E'$ very ample, thus $\text{dim}_{\ \!}H^{0}\big(C_{\lilF},\mathscr{O}(E'-\mathfrak{q})\big)=m$ with basis $\{1,f'_{1},\dots,f'_{m-1}\}$. This implies that the linear combination
	$$
	\mathcal{F}_{\!\mathbf{A}'}
	\ \ \Def\ \ 
	f'+A'_{0}+\sum_{i=1}^{m-1}A'_{i}\ f'_{i}
	$$
is the generic element of the interval $I(f',E')$. By Remark \ref{remark: regarding the star condition}, $E'-\mathfrak{q}$ is effective and very ample. Hence \eqref{equation: new inequality} and Lemmas \ref{lemma: primality of (F_A)} and \ref{lemma: separability of (F_A)} provide us with a Galois group $\text{Gal}\big(\mathcal{F}_{\!\mathbf{A}'},\overline{\FF_{\!q}}(\mathbf{A}')\big)$.
	
	Let $R'$ denote the coordinate ring of the affine curve $C_{\lilF}\!\smallsetminus\! E'$. Observe that since $\text{deg}_{\ \!}E'=\text{deg}_{\ \!}E$, the inequalities \eqref{eq: important inequality at E} imply that $\mathfrak{P}_{\lilF}$ lies in $C_{\lilF (\mathbf{A})}\!\smallsetminus\! E'_{\lilF (\mathbf{A})}$. Consider the point $\mathfrak{P}'\overset{{}_{\text{def}}}{=}(\mathcal{F}_{\!\mathbf{A}\!'})\in\text{Spec}_{\ \!} R'(\mathbf{A}\!')$ inside $V(\mathcal{F}_{\mathbf{A}})\subset\sSpec{R'[\mathbf{A}\!']}$. Because Lemma \ref{lemma: separability of (F_A)} says that $\mathfrak{P}'$ is separable, whereas $\mathfrak{h}$ is a codimension-$1$ point in $\A^{\!m+1}_{\lilF}$, the point $\mathfrak{P}'$ corresponds to a discrete valuation on $\kappa(\mathfrak{P}_{\lilF})$. Thus the Galois group $\text{Gal}\big(\text{split}(\mathfrak{P}_{\lilF})\big/\kappa(\mathfrak{P}_{\lilF})\big)$ acts transitively on the roots of any monic polynomial whose roots generate the extension $\text{split}(\mathfrak{P}')\big/\kappa(\mathfrak{P}')$. Because $\text{Gal}\big(\text{split}(\mathfrak{P}_{\lilF})\big/\kappa(\mathfrak{P}_{\lilF})\big)$ is a subgroup of $\text{Gal}\big(\mathcal{F}_{\!\mathbf{A}},\overline{\FF_{\!q}}(\mathbf{A})\big)$, this completes the proof.
\end{proof}

\end{subsection}

%%%%%%%%%%%%%%%%%%%%%%%%%%%%%%%%%%%%%%

\begin{subsection}{Presence of a transposition}\label{subsection: presence of a transposition}
	Fix an algebraic closure $L\Def \overline{\FF_{\!q}(\mathbf{A})}$, and define $L'\subset L$ to be the algebraic closure $L'\Def\overline{\FF_{\!q}(A_{1},\dots,A_{m-1})}$ inside $L$.
	
	Consider the morphism $C_{\FF_{\!q}(\mathbf{A})}\longrightarrow \PP^{1}_{\FF_{\!q}(\mathbf{A})}=\text{Proj}_{\ \!}\FF_{\!q}(\mathbf{A})[t_{0},t_{1}]$. It restricts to the morphism of affine schemes
	\begin{equation}\label{equation: restricted curve map over Omega}
	C_{\FF_{\!q}(\mathbf{A})}\!\smallsetminus\! E_{\FF_{\!q}(\mathbf{A})}\longrightarrow \sSpec{\FF_{\!q}(\mathbf{A})[t]}=D(t_{0})
	\end{equation}
dual to the morphism of $\FF_{\!q}(\mathbf{A})$-algebras $\FF_{\!q}(\mathbf{A})[t]\longrightarrow R_{\FF_{\!q}(\mathbf{A})}$ that takes $t\ \mapsto\ \mathcal{F}_{\!\mathbf{A}}-A_{0}$. Since $E$ is effective and very ample, we can choose a lift $f(\mathbf{x})\in\FF_{\!q}[x_{1},\dots,x_{m}]$ of $f$ as in the proof of Lemma \ref{lemma: separability of (F_A)}, and \eqref{equation: restricted curve map over Omega} becomes the restriction of the morphism
	$$
	\Psi:\sSpec{\FF_{\!q}(\mathbf{A})[x_{1},\dots,x_{m}]}\longrightarrow \sSpec{\FF_{\!q}(\mathbf{A})[t]}
	$$
which takes
	$$
	\mathbf{x}
	\ \ \mapsto\ \ 
	\Psi(\mathbf{x})\ \Def\ f(\mathbf{x})+\sum_{i=1}^{m}A_i\ x_i.
	$$

\begin{proposition}\label{proposition: at most double roots}
	\normalfont
	At each point in $C_{L}\!\smallsetminus\! E_{L}$, the ramification order of $\Psi$ is at most $1$.
\end{proposition}
\begin{proof}
	Let $\Omega^{1}_{C_{L}\!\smallsetminus\! E_{L}}$ denote the $R_{L}$-module of K\"ahler differentials on $C_{L}\!\smallsetminus\! E_{L}=\sSpec{R_{L}}$, and let $d\Psi\in \Omega^{1}_{C_{L}\!\smallsetminus\! E_{L}}$ denote the K\"ahler differential of $\Psi$. In $D(x_{0})$, on a sufficiently small affine open neighborhood $U_{\mathbf{x}}\subset D(x_{0})$ of each point $\mathbf{x}\in C_{L}\!\smallsetminus\! E_{L}$, we have a matrix $M$ as in equation \eqref{equation: key matrix 1}, where the regular functions $r_{i}$ cut out $U_{\mathbf{x}}\cup\big(C_{L}\!\smallsetminus\! E_{L}\big)$. Since $d\Psi=d(\Psi+A_{0})=d\mathcal{F}_{\!\mathbf{A}}$, the ramification divisor of $\Psi$ is the effective divisor corresponding to the subscheme $V\big(\text{det}(M)\big)\cap\big(C_{L}\!\smallsetminus\! E_{L}\big)$ inside $U_{\mathbf{x}}\subset\sSpec{L[x_{1},\dots,x_{m}]}$. The points of $U_{\mathbf{x}}\cap\big(C_{L}\!\smallsetminus\! E_{L}\big)$ where $\Psi$ has ramification order $1$ are exactly the reduced points of $V\big(\text{det}(M)\big)\cap\big(C_{L}\!\smallsetminus\! E_{L}\big)$. Thus is suffices to prove that the $L$-scheme $V\big(\text{det}(M)\big)\cap\big(C_{L}\!\smallsetminus\! E_{L}\big)$ is smooth.
	
	For each $1\le i\le m$, let $M_{mi}$ denote the minor of $M$ that we obtain by removing the $m^{\text{th}}$-row and $i^{\mathrm{th}}$-column of $M$, so that
	$$
	\text{det}(M)
	\ \ =\ \ 
	\sum_{i=1}^{m}(-1)^{m+i}\ \text{det}(M_{mi})\ \Big(\ \frac{\partial f}{\partial x_{i}}+A_{i}\ \Big).
	$$
	From the proof of Lemma \ref{lemma: separability of (F_A)}, we know that $\text{det}(M_{mm})$ is nonzero everywhere on $U_{\mathbf{x}}\cap\big(C_{L}\!\smallsetminus\! E_{L}\big)$. Therefore $\text{det}(M_{mm})$ is invertible on some open neighborhood of $U_{\mathbf{x}}\cap\big(C_{L}\!\smallsetminus\! E_{L}\big)$ inside $U_{\mathbf{x}}$. In this neighborhood, the vanishing locus of $\frac{\text{det}(M)}{\text{det}(M_{mm})}$ coincides with $V\big(\text{det}(M)\big)$. Write
	\begin{equation}\label{equation: modified function}
	\frac{\text{det}(M)}{\text{det}(M_{mm})}
	\ \ =\ \ 
	\mathcal{G}_{\mathbf{A}}+A_{m},
	\end{equation}
	where $\mathcal{G}_{\mathbf{A}}$ is a regular function with no $A_{m}$ dependence. Then $V\big(\text{det}(M)\big)$ is singular at precisely those points where the determinant of the $m\times m$-matrix
	$$
	M'
	\ \ \Def\ \ 
	\left(\!\!
	\begin{array}{ccccc}
	\frac{\partial r_{1}}{\partial x_{1}}
	&
	\cdots
	&
	\frac{\partial r_{1}}{\partial x_{m}}
	\\[2pt]
	\vdots
	&
	\ddots
	&
	\vdots
	\\[2pt]
	\frac{\partial r_{m-1}}{\partial x_{1}}
	&
	\cdots
	&
	\frac{\partial r_{m-1}}{\partial x_{m}}
	\\[6pt]
	& & \ 
	\\[-10pt]
	\!\!\frac{\partial \mathcal{G}_{\mathbf{A}}}{\partial x_{1}}\!\!\!
	&
	\cdots
	&
	\!\!\frac{\partial \mathcal{G}_{\mathbf{A}}}{\partial x_{m}}\!\!\!
	\end{array}
	\!\!\right)
	$$
	vanishes. The absence of $A_{m}$ from $\text{det}(M')$ means that the zeros of $\text{det}(M')$ are defined over the subfield $L'\subset L$, whereas the zeros of \eqref{equation: modified function} are defined over the subfield $\overline{\FF_{\!q}(A_{m})}\subset L$. Because zeros of \eqref{equation: modified function} are not $\overline{\FF_{\!q}}=L'\cap\overline{\FF_{\!q}(A_{m})}$-rational, this completes the proof.
\end{proof}

\begin{proposition}\label{proposition: existence of a ramified point}
\normalfont
	The morphism $\Psi:C_{L}\longrightarrow\PP^{1}_{L}$ is ramified at some point in $C_{L}\!\smallsetminus\! E_{L}$ in each of the following two cases:
	\begin{itemize}
	\item[{\bf (i)}]
	$g>0$;
	\item[{\bf (ii)}]\vskip .2cm
	$\text{deg}_{\ \!}E>1$.
	\end{itemize}
\end{proposition}
\begin{proof}
	The rational function $\mathcal{F}_{\!\mathbf{A}}$ on $C_{L}$ determines a morphism $\mathcal{F}_{\!\mathbf{A}}:C_{L}\longrightarrow\PP^{1}_{L}$. By definition, $\Psi$ and $\mathcal{F}_{\!\mathbf{A}}$ differ by the constant $A_{0}$, and so it suffices to show that $\mathcal{F}_{\!\mathbf{A}}$ is ramified at some point of $C_{L}\!\smallsetminus\! E_{L}$.
	
	At each point $\mathfrak{p}\in C_{L}$, let $\text{ram}_{\mathfrak{p}}(\mathcal{F}_{\!\mathbf{A}})$ denote the ramification order of $\mathcal{F}_{\!\mathbf{A}}$ at $\mathfrak{p}$ (the order of vanishing of the K\"ahler differential $d\mathcal{F}_{\!\mathbf{A}}\in\Omega^{1}_{C_{L}}$ at $\mathfrak{p}$). Define
	$$
	\text{ram}_{C_{L}\!\smallsetminus\! E_{L}}(\mathcal{F}_{\!\mathbf{A}})
	\ \Def\!\!\!
	\underset{{\mathfrak{p}\in C_{L}\!\smallsetminus\! E_{L}}}{\mbox{\larger\larger $\sum$}}\!\!\!\text{ram}_{\mathfrak{p}}(\mathcal{F}_{\!\mathbf{A}})
	\ \ \ \ \ \ \ \ \ \ \ \mbox{and}\ \ \ \ \ \ \ \ \ \ \ 
	\text{ram}_{E_{L}}(\mathcal{F}_{\!\mathbf{A}})
	\ \ \ \Def\!\!\!\!\!
	\underset{{\mathfrak{p}\in \text{supp}_{\ \!}(E_{L})}}{\mbox{\larger\larger $\sum$}}\!\!\!\!\!\text{ram}_{\mathfrak{p}}(\mathcal{F}_{\!\mathbf{A}}).
	$$
Then $\text{ram}_{C_{L}}(\mathcal{F}_{\!\mathbf{A}})=\text{ram}_{C_{L}\!\smallsetminus\! E_{L}}(\mathcal{F}_{\!\mathbf{A}})+\text{ram}_{E_{L}}(\mathcal{F}_{\!\mathbf{A}})$. Recall that $k=\text{deg}\big(\text{div}(f)_{-}\big)$. Lemmas \ref{lemma: primality of (F_A)} and \ref{lemma: separability of (F_A)} imply that the morphism $\mathcal{F}_{\!\mathbf{A}}:C_{L}\longrightarrow\PP^{1}_{L}$ is finite and separable, so satisfies Riemann-Hurwitz \cite[\S IV, Corollary 2.4]{Hartshorne}. Since $k$ is the degree of $\mathcal{F}_{\!\mathbf{A}}$, this gives
	$$
	2\ (g+k-1)
	\ \ =\ \ 
	\text{ram}_{C_{L}\!\smallsetminus\! E_{L}}(\mathcal{F}_{\!\mathbf{A}})+\text{ram}_{E_{L}}(\mathcal{F}_{\!\mathbf{A}}).
	$$
Thus it suffices to show that
	\begin{equation}\label{equation: ramification inequality}
	\text{ram}_{E_{L}}(\mathcal{F}_{\!\mathbf{A}})
	\ \ <\ \ 
	2\ (g+k-1).
	\end{equation}
	
	Fix a point $\mathfrak{p}\in\text{supp}_{\ \!}(E_{L})$, and fix a uniformizing parameter $z$ in the stalk $\mathscr{O}_{C_{L},\mathfrak{p}}$. Let $m_{\mathfrak{p}}$ denote the order of $E_{L}$ at $\mathfrak{p}$, and recall that $k_{\mathfrak{p}}$ denotes the degree of the pole of $f$ at $\mathfrak{p}$. Because $E$ is effective, our assumption \eqref{eq: important inequality at E} implies that $k_{\mathfrak{p}}\ge m_{\mathfrak{p}}>0$. Because $E$ is very ample, $H^{0}\big(C_{L},\mathscr{O}(E_{L})\big)$ is basepoint free, and thus there exists some nontrivial $\FF_{\!q}$-linear combination $\widetilde{A}_{0}$ of the variables $A_{0},\dots,A_{m}$ so that we can write the rational function $\sum_{i=0}^{m}A_{i}\ f_{i}$ on $C_{L}$ as
	$$
	\sum_{i=0}^{m}A_{i}\ f_{i}\big(\tfrac{1}{z}\big)
	\ =\ 
	\widetilde{A}_{0}+\mathcal{G}_{\!\mathbf{A}}\big(\tfrac{1}{z}\big),
	$$
where the order of the pole of $\mathcal{G}_{\mathbf{A}}\big(\frac{1}{z}\big)$ at $\mathfrak{p}$ is between 1 and $m_{\mathfrak{p}}$. Write
	$$
	\mathcal{F}_{\!\mathbf{A}}
	\ \ =\ \ 
	\frac{1}{z^{k_{\mathfrak{p}}}}
	\ 
	\big(
	\ 
	\widetilde{f}(z)+\widetilde{A}_{0}\ z^{k_{\mathfrak{p}}}+\widetilde{\mathcal{G}}_{\mathbf{A}}(z)
	\ 
	\big),
	$$
where $\widetilde{f}(z)$ is an $\FF_{\!q}$-rational function that does not vanish at $z=0$, and where the order of vanishing of $\widetilde{\mathcal{G}}_{\mathbf{A}}(z)$ at $z=0$ is between $k_{\mathfrak{p}}-1$ and $k_{\mathfrak{p}}-m_{\mathfrak{p}}$. Thus the order of vanishing of $d\big(\frac{1}{\mathcal{F}_{\!\mathbf{A}}}\big)$ at $z=0$ is equal to the order of vanishing of the function
	$$
	k_{\mathfrak{p}}\ z^{k_{\mathfrak{p}}-1}\ \big(\ \widetilde{f}(z)+\widetilde{A}_{0}\ z^{k_{\mathfrak{p}}}+\widetilde{\mathcal{G}}_{\mathbf{A}}(z)\ \big)
	\ +\ z^{k_{\mathfrak{p}}}\ \Big(\ \frac{d\widetilde{f}}{dz}+k_{\mathfrak{p}}\ \widetilde{A}_{0}\ z^{k_{\mathfrak{p}}-1}+\frac{d\widetilde{\mathcal{G}}_{\mathbf{A}}}{dz}\ \Big)
	$$
at $z=0$. This implies that:
	\begin{itemize}
	\item[$\bullet$]
	If $\text{char}_{\ \!}\FF_{\!q}$ does not divide $k_{\mathfrak{p}}$, then $\text{ram}_{\mathfrak{p}}(\mathcal{F}_{\!\mathbf{A}})=k_{\mathfrak{p}}-1$;
	\item[$\bullet$]\vskip .2cm
	If $\text{char}_{\ \!}\FF_{\!q}$ divides $k_{\mathfrak{p}}$, then $\text{ram}_{\mathfrak{p}}(\mathcal{F}_{\!\mathbf{A}})\le 2k_{\mathfrak{p}}-2$.\vskip .2cm
	\end{itemize}
Repeating this argument at all points $\mathfrak{p}$ in $\text{supp}_{\ \!}(E_{L})$, we see that
	$$
	\text{ram}_{E_{L}}(\mathcal{F}_{\!\mathbf{A}})
	\ \ \le\ \ 
	2k-2\ \#\text{supp}(E_{L}).
	$$
Thus \eqref{equation: ramification inequality} is satisfied whenever $g>0$ or $\text{deg}_{\ \!}E>1$.
\end{proof}

\begin{proposition}\label{proposition: system of equations has no solution}
\normalfont
	Assume that one of the following two conditions holds:
	\begin{itemize}
	\item[{\bf (a)}]
	There exists a very ample effective divisor $E_{0}$ on $C$ such that $E\ge 3E_{0}$;
	\item[{\bf (b)}]\vskip .1cm
	There exists a very ample effective divisor $E_{0}$ on $C$ such that $E\ge 2E_{0}$, $\text{char}_{\ \!}\FF_{\!q}=2$, and $df|_{C\!\smallsetminus\! E}$ vanishes at a nonempty finite set.
	\end{itemize}
Then the morphism $\Psi:C_{L}\longrightarrow\PP^{1}_{L}$ {\em separates critical points} in $C_{L}\!\smallsetminus\! E_{L}$, i.e., there do not exist distinct points $x,y\in C_{L}\!\smallsetminus\! E_{L}$ satisfying the system of equations
	\begin{equation}\label{equation: separating critical system}
	\begin{array}{rcll}
	d\Psi|_{x}
	&
	\!\!\!\!=\!\!\!\!
	&
	0
	\\[10pt]
	d\Psi|_{y}
	&
	\!\!\!\!=\!\!\!\!
	&
	0
	\\[10pt]
	\Psi(x)
	&
	\!\!\!\!=\!\!\!\!
	&
	\Psi(y).
	\end{array}
	\end{equation}
\end{proposition}
\begin{proof}
	It suffices to prove that the morphism $\mathcal{F}_{\!\mathbf{A}}:C_{L}\longrightarrow\PP^{1}_{L}$ separates critical points. Assume that $E\ge n\ \!E_{0}$, with $E_{0}$ a very ample effective divisor on $C_{L}$, and with $n=2$ or $3$. Let $m_{0}=\text{dim}_{\ \!}H^{0}\big(C,\mathscr{O}(E_{0})\big)-1$. Interpret $C$ as a closed subvariety of $\PP^{m_{0}}$ via the closed embedding provided by $E_{0}$. The standard proof of Bertini's Theorem \cite[\S II.8, proof of Theorem 8.18]{Hartshorne} implies that for any two distinct points $x,y\in C_{\lilF}\!\smallsetminus\! E_{\lilF}$, we can choose a linear form $t$ on $\PP^{m_{0}}_{\lilF}$ whose restriction to $C_{\lilF}$ provides local uniformizing parameters $t-t(x)$ at $x$ and $t-t(y)$ at $y$. We can furthermore choose $t$ so that it satisfies the generic condition
	\begin{equation}\label{equation: condition L2}
	t(x)\ \ne\ t(y).
	\end{equation}
	
	Since $t\in H^{0}\big(C_{\lilF},\mathscr{O}(E_{0,\lilF})\big)$ and $E\ge n E_{0}$, we have $1,t,t^{2},\dots,t^{n}\in H^{0}\big(C_{\lilF},\mathscr{O}(E_{\lilF})\big)$.
Choose a new basis $\{1,g_{1},g_{2},\dots,g_{m}\}$ of $H^{0}\big(C_{\lilF},\mathscr{O}(E_{\lilF})\big)$ such that $g_{i}=t^{i}$ for $0\le i\le n$. Let $\{B_{0},B_{1},B_{2},\dots\}$ denote linear generators of $\overline{\FF_{\!q}}[\mathbf{A}]$ in this new basis, with $B_{0}=A_{0}$. Then
	\begin{equation}\label{equation: special decomposition of F_A}
	\mathcal{F}_{\!\mathbf{A}}
	\ =\ 
	f+A_{0}+B_{1}\ \!t+B_{2}\ \!t^{2}+\cdots+B_{n}\ \!t^{n}+\mathcal{G}_{\mathbf{A}},
	\end{equation}
where $\mathcal{G}_{\mathbf{A}}=\sum_{i=n+1}^{m}B_{i}\ \!g_{i}$. Define
	$$
	\Phi
	\ \ \overset{{}_{\text{def}}}{=}\ \ 
	\mathcal{F}_{\!\mathbf{A}}-B_{1}\ \!t-B_{2}\ \!t^{2}.
	$$
	
	Again by the dimension counts in \cite[\S II.8, proof of Theorem 8.18]{Hartshorne}, we can fix a Zariski open neighborhood $U\subset C_{\lilF}\!\smallsetminus\! E_{\lilF}$ containing both $x$ and $y$, such that the restriction of $t$ to $C_{\lilF}\!\smallsetminus\! E_{\lilF}$ provides a uniformizing parameter $t-t(u)$ at every point $u\in U$. Define
	$$
	U_{xy}
	\ \overset{{}_{\text{def}}}{=}\ 
	\big(U\times_{\lilF}U\big)\!\smallsetminus\!\big\{\text{diagonal\ in\ }C_{\lilF}\times C_{\lilF}\big\}.
	$$
At each $L$-valued point $(u,v)$ in $U_{xy}$, the system of equations \eqref{equation: separating critical system} holds for the function $\mathcal{F}_{\!\mathbf{A}}$ if and only if $(u,v)$ satisfies the single $\overline{\FF_{\!q}}$-valued matrix equation
	\begin{equation}\label{equation: separating matrix equation}
	\left(
	\begin{array}{cc}
	1
	&
	2\ t(u)
	\\[6pt]
	1
	&
	2\ t(v)
	\\[6pt]
	t(v)-t(u)
	&
	t(v)^{2}-t(u)^{2}
	\end{array}
	\right)
	\left(
	\begin{array}{c}
	B_{1}
	\\
	B_{2}
	\end{array}
	\right)
	\ =\ 
	\left(
	\begin{array}{c}
	-\frac{d\Phi}{dt}(u)
	\\[10pt]
	-\frac{d\Phi}{dt}(v)
	\\[10pt]
	\Phi(u)-\Phi(v)
	\end{array}
	\right),
	\end{equation}
where $\frac{d\Phi}{dt}$ denotes the regular function on $U$ such that $d\Phi=\frac{d\Phi}{dt}dt$ as a global section of $\Omega^{1}_{U}$. Define functions $\varphi$ on $U$ and $c$ on $U_{xy}$ according to
	$$
	\varphi
	\ \overset{{}_{\text{def}}}{=}\ 
	-\ \!\frac{d\Phi}{dt}
	\ \ \ \ \ \ \ \ \ \ \mbox{and}\ \ \ \ \ \ \ \ \ 
	c(u,v)
	\ \overset{{}_{\text{def}}}{=}\ 
	\Phi(u)-\Phi(v).
	$$
By \eqref{equation: condition L2}, the $3\!\times\!2$-matrix at left in \eqref{equation: separating matrix equation} has rank $2$ everywhere in $U_{xy}$. Hence \eqref{equation: separating matrix equation} holds at $(u,v)$ if and only if $(u,v)$ satisfies the single determinant equation
	\begin{equation}\label{equation: separating determinant equation}
	\text{det}
	\left(
	\begin{array}{ccc}
	1
	&
	2\ t(u)
	&
	\varphi(u)
	\\[10pt]
	1
	&
	2\ t(v)
	&
	\varphi(v)
	\\[10pt]
	\ t(v)-t(u)\ 
	&
	\ \ t(v)^{2}-t(u)^{2}\ \ 
	&
	\ c(u,v)\ 
	\end{array}
	\right)
	\ \ =\ \ 
	0.
	\end{equation}
	
	Interpret \eqref{equation: separating determinant equation} as an equation over $\overline{\FF_{q}}$ in the variables $u,v,A_{1},...,A_{m}$. Let $T\subset \A^{\!m}_{\lilF}\times_{\lilF}U_{xy}$ denote the subscheme cut out by this equation, with projections
	$$
	\xymatrix{\A^{\!m}_{\lilF} & T \ar@{->>}[r]^{\text{pr}_{2}} \ar@{->>}[l]_{\text{pr}_{1}} & U_{xy}}.
	$$
Letting $\eta$ denote the generic point of $\A^{\!m}_{\lilF}$, it suffices to prove that the fiber $\text{pr}^{-1}_{1}(\eta)\subset T$ is empty. Because the leftmost matrix in \eqref{equation: separating matrix equation} has rank $2$, each fiber of $\text{pr}_{2}$ in $T$ is at most $(m-2)$-dimensional. The generic fiber $\text{pr}^{-1}_{1}(\eta)$ of $T$ is cut out by a single equation in the $2$-dimensional space $(U_{xy})_{\overline{\FF_{q}}(\mathbf{A})}$. Thus if the determinant in \eqref{equation: separating determinant equation} is not constantly equal to $0$, we have
	$$
	\text{dim}_{\ \!}T
	\ \le\ 
	1+m-2
	\ <\ 
	m,
	$$
which implies that the image $\text{pr}_{1}(T)\subset\A^{\!m}_{\lilF}$ cannot contain the generic point of $\A^{\!m}_{\lilF}$. In order to show that \eqref{equation: separating determinant equation} has no solutions in $U_{xy}$, it thus remains to show that the determinant appearing in \eqref{equation: separating determinant equation} is not constantly equal to $0$.
	
	Let $d(u,v)$ be the determinant that appears in \eqref{equation: separating determinant equation}. A straightforward calculation gives
	$$
	d(u,v)
	\ =\ 
	\big(\ \!t(v)-t(u)\ \!\big)\ \Big(\ 2\ \!c(u,v) +\big(\ \!t(v)-t(u\ \!)\big)\ \big(\ \!\varphi(u)+\varphi(v)\ \!\big)\ \Big).
	$$
If $n\geq 3$, then the coefficient of $B_3$ in $2\ \!c(u,v) +\big(\ \!t(v)-t(u\ \!)\big)\ \!\big(\ \!\varphi(u)+\varphi(v)\ \!\big)$ is
	$$
	2\ \!\big(\ \!t(u)^{3}-t(v)^{3}\ \!) + 3\ \!\big(\ \!t(v)^{2}-t(u)^{2}\ \!\big).
	$$
By \eqref{equation: condition L2}, this last expression is nonzero in any characteristic. If $n=2$ and $\text{char}_{\ \!}\FF_{q}=2$, then
	$$
	d(u,v)
	\ =\ 
	\big(\ \!t(v)-t(u)\ \!\big)^{2}\ \big(\ \!\varphi(u)+\varphi(v)\ \!\big).
	$$
If $df$ is nonconstant in this case, then $\varphi(u)+\varphi(v)=\varphi(u)-\varphi(v)$ is not constantly zero.

	Because the Zariski open subsets $U_{xy}$ cover $\big(C_{\lilF}\times C_{\lilF}\big)\!\smallsetminus\!\{\text{diagonal}\}$ as $(x,y)$ varies inside $\big(C_{\lilF}\times C_{\lilF}\big)\!\smallsetminus\!\{\text{diagonal}\}$, this completes the proof.
\end{proof}

\begin{corollary}\label{corollary: contains a transposition}
\normalfont
	Assume that one of the following two conditions holds:
	\begin{itemize}
	\item[{\bf (a)}]
	There exists a very ample effective divisor $E_{0}$ on $C$ such that $E\ge 3E_{0}$;
	\item[{\bf (b)}]\vskip .1cm
	There exists a very ample effective divisor $E_{0}$ on $C$ such that $E\ge 2E_{0}$, $\text{char}_{\ \!}\FF_{\!q}=2$, and $df|_{C\!\smallsetminus\! E}$ vanishes on a nonempty finite set.
	\end{itemize}
Then the subgroup $\text{Gal}\big(\mathcal{F}_{\!\mathbf{A}}\big/\overline{\FF_{\!q}}(\mathbf{A})\big)\subset S_{k}$ contains a transposition.
\end{corollary}
\begin{proof}
	If $g=0$, then each of the conditions (a) and (b) implies condition (ii) of Proposition \ref{proposition: existence of a ramified point}. Thus for any $g$, the morphism $\Psi$ is ramified at some closed point of $C_{L}\!\smallsetminus\! E_{L}$. Let $\alpha$ be such a point, which is to say that the morphism $\Psi:C_{L}\!\smallsetminus\! E_{L}\longrightarrow\A^{1}_{L}$ is ramified at $\alpha$. Proposition~\ref{proposition: at most double roots} says that the order of ramification at any point in $C_{L}\!\smallsetminus\! E_{L}$ is at most $1$. Thus the factorization type of the fiber of $\Psi$ containing $\alpha$ is $(2,1,\dots,1)$. As Proposition~\ref{proposition: system of equations has no solution} says that the critical values of $\Psi$ are distinct, this implies that $\Psi(\mathbf{x})=\Psi(\alpha)$ has at least $k-1$ solutions. However, since $\alpha$ is a ramification point, the fiber over $\Psi(\alpha)$ has exactly one double point. Hence the inertia group over $\Psi(\alpha)$ permutes two factors of $\mathcal{F}_{\!\mathbf{A}}=\Psi(\mathbf{x})+A_0$ and fixes all others. Thus $\text{Gal}\big(\mathcal{F}_{\!\mathbf{A}}\big/\overline{\FF_{\!q}}(\mathbf{A})\big)$ contains a transposition.
\end{proof}

\begin{proof}[{\it Completion of the proof of Theorem \ref{theorem: Galois group is S_k}}.]
	By Remark \ref{remark: regarding the star condition}, if either of the conditions (a) or (b) holds, then Proposition \ref{proposition: doubly transitive on points} holds and $\text{Gal}\big(\mathcal{F}_{\!\mathbf{A}}\big/\overline{\FF_{\!q}}(\mathbf{A})\big)$ is doubly transitive. By Corollary \ref{corollary: contains a transposition}, $\text{Gal}\big(\mathcal{F}_{\!\mathbf{A}}\big/\overline{\FF_{\!q}}(\mathbf{A})\big)$ contains a transposition. By Lemma \ref{lemma: serre S_k}, we have $\text{Gal}\big(\mathcal{F}_{\!\mathbf{A}}\big/\overline{\FF_{\!q}}(\mathbf{A})\big)\cong S_{k}$.
\end{proof}

\end{subsection}

\end{section}

\vskip .5cm

%%%%%%%%%%%%%%%%%%%%%%%%%%%%%%%%%%%%%%
%%%%%%%%%%%%%%%%%%%%%%%%%%%%%%%%%%%%%%

\begin{section}{Proof of Theorem \ref{theorem: main theorem}}\label{section: counting argument}
	
	We now use the Galois group calculation in \S\ref{section: calculation of the Galois group} to prove Theorem \ref{theorem: main theorem}.

\begin{subsection}{Setup for the proof of Theorem \ref{theorem: main theorem}}\label{subsec: factorization type}
	Let $\mathcal{F}_{\!\mathbf{A}}\in R[\mathbf{A}]$ be the element defined in \eqref{eq: formula for generic element of I}, and let $\text{pr}_{1}:V(\mathcal{F}_{\!\mathbf{A}})\longrightarrow\A^{\!m+1}$ denote the resulting projection. For each $\FF_{\!q}$-rational point $\mathbf{a}\in\A^{\!m+1}$, let $\mathcal{F}_{\!\mathbf{a}}$ denote the restriction of $\mathcal{F}_{\!\mathbf{A}}$ to $R\cong \kappa(\mathbf{a})\otimes_{\FF_{\!q}[\mathbf{A}]}R[\mathbf{A}]$.

\begin{proposition}\label{prop: deg of Z}
\normalfont
	Let $E$ and $f$ be as in the statement of Theorem \ref{theorem: Galois group is S_k}, and let $Z\subset \A^{\!m+1}_{\FF_{\!q}}$ denote the branch locus of $\text{pr}_{1}:V(\mathcal{F}_{\!\mathbf{A}})\longrightarrow\A^{\!m+1}$. Then $Z$ is pure of codimension $1$ in $\A^{\!m+1}$, and it satisfies the inequality
	\begin{equation}\label{eq:deg Z}
	\deg Z
	\ \le\  
	g-2+2k,
	\end{equation}
where $g$ denotes the genus of $C$.
\end{proposition}
\begin{proof}
	Define $V\overset{{}_{\text{def}}}{=}H^{0}\big(C,\mathscr{O}\big(\text{div}(f)_{-}\big)\big)^{\ast}$, the $\FF_{\!q}$-vector space dual of the space of rational functions $H^{0}\big(C,\mathscr{O}\big(\text{div}(f)_{-}\big)\big)$, and consider the pair of dual projective spaces
	$$
	\PP(V)
	\ \overset{{}_{\text{def}}}{=}\ 
	\text{Proj}\ \text{Sym}^{\bullet}_{\FF_{\!q}} V^{\ast}
	\ \ \ \ \ \ \mbox{and}\ \ \ \ \ \ 
	\PP(V^{\ast})
	\ \overset{{}_{\text{def}}}{=}\ 
	\text{Proj}\ \text{Sym}^{\bullet}_{\FF_{\!q}} V,
	$$
	where $\text{Sym}^{\bullet}_{\FF_{\!q}}V^{\ast}$ and $\text{Sym}^{\bullet}_{\FF_{\!q}}V$ denote the graded symmetric $\FF_{\!q}$-algebras on $V^{\ast}$ and $V$, respectively. Because $E$ is a very ample effective divisor, the inequalities \eqref{eq: important inequality at E} imply that $\text{div}(f)_{-}$ is very ample and effective. Identify $C$ with its image under the closed embedding
	\begin{equation}\label{equation: closed embedding for Theorem B}
	C\mono\PP(V)
	\end{equation}
induced by $\text{div}(f)_{-}$. Pass to the algebraic closure $\overline{\FF_{\!q}}$ to obtain a smooth, closed, irreducible subvariety $C_{\lilF}\subset\PP(V_{\lilF})$. By \cite[\S3.1.3 \& \S5.1]{SGA7_XVII}, this subvariety determines a dual variety $C^{\vee}_{\lilF}\ \subset\ \PP(V^{\ast}_{\lilF})$.
	
	We claim that $C^{\vee}_{\lilF}$ is a hypersurface in $\PP(V^{\ast}_{\lilF})$. To see this, let $\mathscr{N}$ denote the conormal sheaf on $C_{\lilF}$ in $\PP(V)$, and let $\PP(\mathscr{N})$ denote its associated projective scheme over $C_{\lilF}$, which comes with a projection
	\begin{equation}\label{equation: projection from P(N)}
	\PP(\mathscr{N})
	\epi
	\PP(V^{\ast}_{\lilF})
	\end{equation}
	(see \cite[\S3.1]{SGA7_XVII} for details). Note that
	$
	\text{dim}_{\ \!}\PP(\mathscr{N})
	=
	\text{dim}_{\ \!}\PP(V_{\lilF})-1
	$. By \cite[Proposition 3.5]{SGA7_XVII}, if the projection \eqref{equation: projection from P(N)} is {\em not} everywhere ramified, then the projection \eqref{equation: projection from P(N)} induces a birational morphism $\PP(\mathscr{N})\!\xymatrix{{}\ar@{-->}[r]&{}}\!C^{\vee}_{\lilF}$. Thus $C^{\vee}_{\lilF}$ is a hypersurface as soon as \eqref{equation: projection from P(N)} is not everywhere ramified. By \cite[Proposition 3.3]{SGA7_XVII}, exhibiting a point of $\PP(\mathscr{N})$ where \eqref{equation: projection from P(N)} is unramified reduces to exhibiting a hyperplane $H\subset\PP(V_{\lilF})$ and a point $x_{0}$ of the scheme-theoretic intersection $C_{\lilF}\cap H$ such that $x_{0}$ is a {\em non-degenerate} (or {\em ordinary}) {\em quadratic singularity} of $C_{\lilF}\cap H$ (see \cite[\S1.1]{SGA7_XVII} for details). When $C_{\lilF}\cap H$ is $0$-dimensional, as it is in our case, the condition that a point $x_{0}$ in $C_{\lilF}\cap H$ be a non-degenerate quadratic singularity reduces to the condition that the component of $C_{\lilF}\cap H$ containing $x_{0}$ is isomorphic to $\text{Spec}_{\ \!}\overline{\FF_{\!q}}[t]\big/(t^{2})$. Our ability to find a hyperplane $H\subset\PP(V_{\lilF})$ and point $x_{0}\in C_{\lilF}\cap H$ satisfying this condition follows from the decomposition \eqref{equation: special decomposition of F_A} of $\mathcal{F}_{\!\bold{A}}$ provided in the proof of Proposition \ref{proposition: system of equations has no solution}. Indeed, choose values $a_{0},b_{1}\in\overline{\FF_{\!q}}$, for $A_{0}$ and $B_{1}$ in \eqref{equation: special decomposition of F_A}, so that $f+a_{0}+b_{1}t+\mathcal{G}_{\!\bold{A}}$ vanishes to order $\ge2$ at a fixed $\overline{\FF_{\!q}}$-valued point $x_{0}$ in $C_{\lilF}$, and then choose the value $b_{2}=1$ for $B_{2}$.
	
	Thus $C^{\vee}_{\lilF}\subset\PP(V)$ is a hypersurface, and the hypotheses of \cite[\S 5.2]{SGA7_XVII} hold. By \cite[Proposition 5.7.2]{SGA7_XVII}, we then have
	\begin{equation}\label{equation: degree formula for C^vee}
	\text{deg}_{\ \!}C^{\vee}_{\lilF}
	\ \ =\ \ 
	g-2+2\ \!k.
	\end{equation}
	
	Our parameter space $\A^{\!m+1}_{\lilF}$ admits a natural identification with a distinguished affine open chart inside a linear subspace $L\subset\PP(V^{\ast})$. Because the hyperplane $H_{\bold{a}}$ associated to a point $\bold{a}\in\A^{\!m+1}$ does not intersect $C_{\lilF}$ at $\text{supp}_{\ \!}E$, the morphism $V(\mathcal{F}_{\!\bold{A}})_{\lilF}\epi\A^{\!m+1}_{\lilF}$ is ramified over $\bold{a}\in\A^{\!m+1}_{\lilF}$ if and only if $\bold{a}\in C^{\vee}_{\lilF}\cap L$ \cite[\S3.1.3]{SGA7_XVII} \cite[\S17.13.7 \& Proposition 17.13.2]{EGAIV_IV}. By Lemma \ref{lemma: separability of (F_A)}, $V(\mathcal{F}_{\!\bold{A}})$ is generically unramified over $\A^{\!m+1}_{\lilF}$, thus the scheme-theoretic intersection $C^{\vee}_{\lilF}\cap L$ has dimension strictly less than $\text{dim}_{\ \!}L$, which is to say that the intersection is proper \cite[Definition 7.1]{Fulton}. Thus $C^{\vee}_{\lilF}\cdot L$ is pure of codimension $1$ in $L$, and B\'ezout's Theorem in $\PP(V^{\ast}_{\lilF})$ \cite[Proposition 8.4]{Fulton} combined with \eqref{equation: degree formula for C^vee} implies that
	$$
	\text{deg}_{\ \!}C^{\vee}_{\lilF}\cdot L
	\ \ =\ \ 
	g-2+2\ \!k.
	$$
Because $Z_{\lilF}=C^{\vee}_{\lilF}\cap\A^{\!m+1}_{\lilF}\subset C^{\vee}_{\lilF}\cap L$, with $\text{deg}_{\ \!}Z=\text{deg}_{\ \!}Z_{\lilF}$, the formula \eqref{eq:deg Z} follows.
\end{proof}

\begin{remark} 
Suppose that $\bold{a}$ is an $\FF_{\!q}$-rational point in $\A^{\!m+1}$ such that $R\big/(\mathcal{F}_{\!\mathbf{a}})$ is a separable $\FF_{\!q}$-algebra. Then since $R$ is a Dedekind domain, the ideal $(\mathcal{F}_{\!\mathbf{a}})$ can be written uniquely as
	\begin{equation*}
	(\mathcal{F}_{\!\mathbf{a}})
	\ =\ 
	\mathfrak{f}_{1}\cdots\mathfrak{f}_{\ell},
	\end{equation*} 
where the $\mathfrak{f}_{i}$ are distinct prime ideals in $R$, with each $k(\mathfrak{f}_{i})=R/(\mathfrak{f}_{i})$ a separable extension of $\FF_{\!q}$. Note that in this case, we have
	\begin{equation}\label{equation: partition for factorization type}
	k
	\ =\ 
	\deg_{\ \!}\mathfrak{f}_{1}+\cdots +\deg_{\ \!}\mathfrak{f}_{\ell}.
	\end{equation}
\end{remark}

\begin{definition}\label{definition: factorization type}\normalfont
	If $\bold{a}$ is an $\FF_{\!q}$-rational point in $\A^{\!m+1}$, then the \textit{factorization type} $\lambda_{\mathbf{a}}$ is the partition of $k$ given in \eqref{equation: partition for factorization type}.
	
	The {\em factorization type counting function} for a fixed partition $\lambda$ of $k$ is the assignment $\pi_{C}(-;\lambda)$ taking the short interval $I(f,E)$ to the value
	$$
	\pi_{C}\big(I(f,E);\lambda\big)
	\ \ \Def\ \ 
	\#
	\big\{
		\mathbf{a}\in \A^{\!m+1}(\FF_{\!q}):R/(\mathcal{F}_{\bold{a}})\mbox{\ is separable\ and\ }\lambda_{\bold{a}}=\lambda
	\big\}.
	$$
\end{definition}

\begin{definition}\label{definition: probability}\normalfont
	Given a permutation $\sigma\in S_{k}$, its {\em partition type}, denoted $\lambda_{\sigma}$, is the partition of $k$ determined by the cycle decomposition of $\sigma$. For an arbitrary partition $\lambda$ of $k$, we define 
	\begin{equation}
	P(\lambda)
	\ \overset{{}_{\text{def}}}{=}\ 
	\frac{\# \{\sigma\in S_k \lvert \lambda_{\sigma}=\lambda\}}{k!}.
	\end{equation}
In other words, $P(\lambda)$ is the probability that a given  permutation in $S_k$ has partition type $\lambda$.
\end{definition}

\end{subsection}

%%%%%%%%%%%%%%%%%%%%%%%%%%%%%%%%%%%%%%

\begin{subsection}{Proof of the main theorem}\label{subsection: Proof of Theorem A}
	\ We begin by proving a general theorem that provides an estimate for the number of $\FF_{\!q}$-rational substitutions in the variables $A_{0},\dots,A_{1}$ for which the regular function $f+A_{0}+A_{1}f_{1}+\cdots A_{m}f_{m}$ on $C\!\smallsetminus\! E$ factors according to a given partition of $k$. The formulation of this theorem, as well as its proof, is very much in the spirit of \cite[Proposition 3.1]{BBR15}.
	
\begin{theorem}\label{theorem: main theorem for arbitrary partitions}
\normalfont
	Let $C$ be a smooth projective geometrically irreducible curve over $\FF_{\!q}$ of arithmetic genus $g$. Fix a positive integer $k$. Then there exists a constant $c(k,g)>0$, depending only on $k$ and $g$, such that for any datum consisting of
	\begin{itemize}
	\item[{\bf (i)}]
	a partition $\lambda$ of $k$;
	\item[{\bf (ii)}]\vskip .1cm
	a prime number $p$ and a power $q=p^{e}$;
	\item[{\bf (iii)}]\vskip .1cm
	an effective divisor $E$ on $C$ and a regular function $f$ on $C\!\smallsetminus\! E$ satisfying
	\begin{equation}\label{equation: two important inequalities}
	E
	\ <\ 
	\text{div}(f)_{-}
	\ \ \ \ \ \ \ \ \ \mbox{and}\ \ \ \ \ \ \ \ \ 
	k
	\ \overset{{}_{\text{def}}}{=}\ 
	\deg_{\ \!}\text{div}(f)_{-},
	\end{equation}
such that $p$, $E$, and $f$ satisfy either of the following conditions:
		\begin{itemize}
		\item[{\bf (a)}]\vskip .1cm
		There exists a very ample effective divisor $E_{0}$ on $C$ with $\deg E_0 \geq 2g+1$ such that $E\ge 3E_{0}$;
		\item[{\bf (b)}]\vskip .1cm
		There exists a very ample effective divisor $E_{0}$ on $C$ with $\deg E_0 \geq 2g+1$ such that $E\ge 2E_{0}$, $p=2$, and $df|_{C\smallsetminus E}$ vanishes on a nonempty finite set,
		\end{itemize}
		\vskip .2cm
	\end{itemize}
we have
	\begin{equation}\label{equation: crucial inequality}
	\Big| \pi_{C}\big(I(f,E);\lambda\big)-P(\lambda)\ \!q^{m+1}\Big|
	\ \leq\ 
	c(m,k)\ q^{m+\frac{1}{2}},
	\end{equation}
where $m\overset{{}_{\text{def}}}{=}\text{deg}(E)-g$.
\end{theorem}

\begin{proof}[{\it Proof of Theorem \ref{theorem: main theorem for arbitrary partitions}}]
		By Theorem~\ref{theorem: Galois group is S_k}, we have that $\Gal(\mathcal{F}_{\mathbf{A}}, \overline{\FF}_q(\mathbf{A})) = S_k$. Note also that by the Riemann-Roch Theorem, the requirement that $\text{deg}_{\ \!}E_0\ge 2g+1$ in (a) and (b) implies that $$\text{dim}_{\ \!}H^{0}\big(C,\mathscr{O}(E)\big)=\text{deg}(E)-g+1=m+1.$$
%	Because $\text{div}(f)_{-}>E$, this implies that
%		\begin{equation}\label{equation: inequality between k and m}
%		k
%		\ >\ 
%		\text{deg}_{\ \!}E
%		\ =\ 
%		m+g
%		\ \ge\ 
%		m.
%		\end{equation}
		
	Let $Z$ be the branch locus of the morphism $V(\mathcal{F}_{\mathbf{A}})\lra \A^{m+1}$ as in Proposition \ref{proposition: etale!}. By Proposition \ref{prop: deg of Z}, we have $\deg Z\le g-2+2k$. This provides a bound on both the number of irreducible components of $Z$ and on the degree of each of these irreducible components. Applying Lang-Weil \cite[Theorem~1]{LangWeil1954}, we obtain a constant $c_1(k,g)$, depending only on $k$ and $g$, such that
	\begin{equation}\label{equation: error term}
	\#Z(\FF_q)
	\ \leq\ 
	c_1(k,g)\ q^{m}.
	\end{equation}
		
		Consider the $\FF_q$-varieties $Y\Def V(\mathcal{F}_{\mathbf{A}})_{\A^{m+1}\!\smallsetminus Z}$ and $X\Def \A^{m+1}\!\smallsetminus Z$. By Proposition \ref{proposition: etale!}, the morphism $\rho:=\text{pr}_{1}|_{Y}:Y\lra X$ is finite \'etale of degree $k$.  By the theorem of the primitive element, we can construct the normal closure of the separable extension $\kappa(\mathcal{F}_{\!\bold{A}})\big/\FF_{\!q}(A_{0},\dots,A_{m})$ as the splitting field of some degree-$k$ polynomial over $\FF_{\!q}(A_{0},\dots,A_{m})$. The Galois closure $W$ of $Y$ over $X$ (see \cite[Proposition 5.3.9]{Szamuely}) is isomorphic to the integral closure of the coordinate ring of $X$ in this splitting field, and therefore the Galois group $\text{Aut}_{X}W$ is degree $k$.

	Observe that the closed embedding $C\mono\PP^{m}$ realizes $V(\mathcal{F}_{\!\bold{A}})$ as a hypersurface of degree $k$ inside the afffine open suscheme $\A^{\!2m+1}_{\FF_{\!q}}\subset\A^{\!m+1}_{\FF_{\!q}}\!\times_{\FF_{\!q}}\!\PP^{m}_{\FF_{\!q}}$. Because we can construct $W$ as a connected component of the $k$-fold fiber product $Y\times_{X}\cdots\times_{X}Y$ \cite[Proof of Proposition 5.3.9]{Szamuely}, we can realize $W$ as a locally closed subspace of $\A^{km+k+1}$, whose closure is a hypersurface of degree $\le k^{k}$. Thus we obtain a bound, depending only on $k$, on the degree of the closure of $W$ inside an affine space.
		
	The morphism $W\lra X$ defines a {\em geometric embedding problem}, in the sense of \cite[\S2.1]{BarySoroker2012}. In \cite[Proposition 3.1]{BBR15}, Bary-Soroker, Rosenzweig, and the first author construct a {\em geometric embedding problem associated to a polynomial $\mathcal{F}\in\FF_{\!q}[\bold{A},t]$}, in the sense of \cite[\S2.1, p. 859]{BarySoroker2012}. However, the last two paragraphs of \cite[proof of Proposition 3.1]{BBR15} make no special use of the fact that the geometric embedding is associated to a polynomial. The construction depends only on the following facts:
	\begin{itemize}
	\item[$\bullet$]\vskip .1cm
	the degree of the Galois group of the geometric embedding problem is $k$,
	\item[$\bullet$]\vskip .2cm
	$W$ is a dense open subset of a hypersurface of degree bound by a function of $k$ inside some affine space,
	\item[$\bullet$]\vskip .2cm
	the point count in the branch locus $Z\subset\A^{\!m+1}_{\FF_{\!q}}$ has upper bound \eqref{equation: error term}.
	\end{itemize}
The proof can now proceed exactly as in the last two paragraphs of \cite[proof of Proposition~3.1]{BBR15} upon replacing $V$ in that proof with our variety $X$, and noting that (ii) above lets us replace the constant $c_{2}(m,B)$ appearing in \cite[proof of Proposition~3.1]{BBR15} with a constant depending only on $k$, as in \cite[proof of Theorem 2.3]{BBR15}. Thus we obtain a constant $c(k,g)$, depending only on $k$ and $g$, such that
		\begin{equation}
		\Big| \pi_{C}\big(I(f,E);\lambda\big)-P(\lambda)\ \!q^{m+1}\Big|
		\ \leq\ 
		c(k,g)\ q^{m+\frac{1}{2}},
		\end{equation}
		as desired.
\end{proof}

\begin{proof}[{\it Proof of Theorem \ref{theorem: main theorem}}]
	Use the Young diagram $\young$ to denote the trivial partition of $k$ consisting of a single set. For this partition, we have
	$$
	P\big(\young\big)
	\ =\ 
	\frac{1}{k}
	$$
and $\pi_{C}\big(I(f,E);\young\big)=\pi_{C}\big(I(f,E)\big)$. Because the two possible conditions on $E$ in the statement of Theorem \ref{theorem: main theorem} imply conditions (iii.a) and (iii.b) of Theorem \ref{theorem: main theorem for arbitrary partitions}, the inequality \eqref{equation: crucial inequality} in Theorem \ref{theorem: main theorem for arbitrary partitions} becomes the inequality
	\begin{equation}\label{equation: crucial inequality in case of short interval}
	\Big| \pi_{C}\big(I(f,E)\big)-\frac{q^{m+1}}{k}\Big|
	\ \leq\ 
	c(k,g)\ q^{m+\frac{1}{2}}
	\end{equation}
estimating the number of elements in the short interval $I(f,E)$ with principal divisor irreducible away from $E$. As $I(f,E)=f+H^{0}\big(C,\mathscr{O}(E)\big)$ with $\text{dim}_{\ \!}H^{0}\big(C,\mathscr{O}(E)\big)=m+1$, the asymptotic formula \eqref{equation: asymptotic formula in main theorem} follows immediately from \eqref{equation: crucial inequality in case of short interval}.

	The statement of uniformity in Theorem \ref{theorem: main theorem}, i.e., the statement that the implied constant in the error term $O(q^{-1/2})$ depends only on $k$ and $g$, follows from the fact that the constant $c(k,g)$ in Theorem \ref{theorem: main theorem for arbitrary partitions} depends only on $k$ and $g$.
\end{proof}

\end{subsection}
\end{section}

%%%%%%%%%%%%%%%%%%%%%%%%%%%%%%%%%%%%%%
%%%%%%%%%%%%%%%%%%%%%%%%%%%%%%%%%%%%%%

\bibliography{refrencesEfrat}
\bibliographystyle{plain}

\end{document}